\documentclass[fontsize=11pt,a4paper,reqno,numbers=noendperiod,bibliography=totoc]{scrartcl}
\usepackage[USenglish]{babel}
\usepackage[T1]{fontenc}
\usepackage[latin10]{inputenc} 
\usepackage{fouriernc}
\usepackage{amsmath,amssymb,amsthm}
\usepackage{dsfont,upgreek}
\usepackage[obeyspaces]{url}
\usepackage{colortbl,color,xcolor}
\usepackage{graphicx,tikz,pgfplots}
\usepackage{centernot}
\usepackage{enumitem}
\usepackage{subfigure}
\usepackage{scrhack}
\usepackage{float}
\usepackage{cite}

\AtBeginDocument{
\hypersetup{bookmarks=true,plainpages=false,linktocpage,colorlinks=true,menucolor=yellow,citecolor=green!80!black,linkcolor=red!70!black,filecolor=magenta,urlcolor=magenta,breaklinks,pdfauthor={Thomas Jahn},pdftitle={Orthogonality in Generalized Minkowski Spaces}}
}

\usetikzlibrary{calc,intersections,arrows,patterns,shapes.geometric,shapes.misc,backgrounds}
\usepgfplotslibrary{patchplots,colormaps}

\newcommand{\dd}{\mathrm{d}}
\newcommand{\NN}{\mathds{N}}
\newcommand{\RR}{\mathds{R}}

\newcommand{\fall}{\:\forall\:}
\newcommand{\ex}{\:\exists\:}
\newcommand{\abs}[1]{\left\lvert#1\right\rvert}
\newcommand{\mnorm}[1]{\left\lVert#1\right\rVert}
\newcommand{\setn}[1]{\left\{#1\right\}}
\newcommand{\setcond}[2]{\left\{#1 \:\middle\vert\: #2\right\}}
\newcommand{\defeq}{\mathrel{\mathop:}=}
\newcommand{\lr}[1]{\left(#1\right)}
\newcommand{\mc}[2]{B\!\left(#1,#2\right)}
\newcommand{\ms}[2]{S\!\left(#1,#2\right)}
\newcommand{\p}{\partial}
\newcommand{\strline}[2]{\left\langle#1,#2\right\rangle}
\newcommand{\clray}[2]{\left[#1,#2\right\rangle}
\newcommand{\skpr}[2]{\left\langle#1 \,\middle\vert\, #2\right\rangle}
\newcommand{\skprs}[2]{\langle#1 \,\big\vert\, #2\rangle}
\newcommand{\enquote}[1]{``#1''}
\newcommand{\norel}{\mathrel{\phantom{=}}}
\newcommand{\degr}{\ensuremath{^\circ}}

\newcommand{\mscLink}[1]{\href{http://www.ams.org/mathscinet/msc/msc2010.html?t=#1}{#1}}

\theoremstyle{plain}
\newtheorem{Satz}{Theorem}[section]
\newtheorem{Kor}[Satz]{Corollary}
\newtheorem{Lem}[Satz]{Lemma}
\newtheorem{Prop}[Satz]{Proposition}

\theoremstyle{definition}
\newtheorem{Def}[Satz]{Definition}
\newtheorem{Bsp}[Satz]{Example}
\newtheorem{Bem}[Satz]{Remark}

\theoremstyle{remark}
\newtheorem*{myproof}{Proof of \hyperref[result:birkhoff-best-approximation]{Proposition~\ref*{result:birkhoff-best-approximation}}}

\DeclareMathOperator{\lin}{lin}

\DeclareMathOperator{\bisec}{bsc}
\DeclareMathOperator{\hfl}{hfl}

\let \eps \varepsilon
\let \piup \uppi
\let \subset \subseteq

\addtokomafont{caption}{\small}
\setkomafont{captionlabel}{\bfseries}

\makeatletter
\renewcommand*{\eqref}[1]{%
  \hyperref[{#1}]{\textup{\tagform@{\ref*{#1}}}}%
}
\makeatother
\begin{document}
\parindent 0pt
\title{Orthogonality in Generalized Minkowski Spaces}
\author{Thomas Jahn\\
{\small Faculty of Mathematics, Technische Universit\"at Chemnitz}\\
{\small 09107 Chemnitz, Germany}\\
{\small thomas.jahn\raisebox{-1.5pt}{@}mathematik.tu-chemnitz.de}
}
\date{}
\maketitle

\begin{abstract}
We combine functional analytic and geometric viewpoints on approximate Birkhoff and isosceles orthogonality in generalized Minkowski spaces which are finite-dimensional vector spaces equipped with a gauge. This is the first approach to orthogonality types in such spaces.
\end{abstract}


\textbf{MSC(2010):} \mscLink{46B20}, \mscLink{52A20}, \mscLink{52A21}, \mscLink{52A41}, \mscLink{90C25}

\section{Introduction}
In recent literature, attention was attracted to the investigation of geometric properties of \emph{generalized Minkowski spaces} defined by so-called \emph{gauges}. In finite dimensions, these are positively homogeneous and subadditive functions $\gamma:\RR^d\to\RR$ which have non-negative values and vanish only at the origin $0\in\RR^d$. Obviously, this notion is a generalization of that of a norm because only the homogeneity property has been relaxed. The good news is that many of the concepts of classical functional analysis of finite-dimensional normed spaces still work in generalized Minkowski spaces $(\RR^d,\gamma)$, see \cite{Cobzas2013}. A convenient way for understanding such spaces is to exploit the correspondence between analysis of gauges and the geometry of their unit balls. Examples for this interplay can be found in \cite{JahnKuMaRi2015} and \cite{JahnMaRi2016}. The purpose of the present paper is to add another subject to this list, namely that of orthogonality.

Motivated by the pleasant theory of Hilbert spaces, mathematicians introduced various generalizations of the notion of orthogonality for non-Hilbert spaces. Birkhoff orthogonality is the most popular one, and its usage in the setting of normed spaces reaches from angular measures \cite{ChenLiLu2011}, approximation theory \cite{Hetzelt1985}, curve theory \cite{MartiniWu2012,ShonodaOm2014}, orthocentric systems \cite{PachecoSo2015}, matrix theory \cite{BenitezFeSo2007,BhattacharyyaGr2013,Sain2017,GhoshPaSa2016a,GhoshPaSa2016b,HaitPaSa2015,PaulSa2013a,PaulSa2013b,MalPaSa2017b}, and orthogonal decompositions of Banach spaces \cite{Alber2005,JhaPaSa2013,BlagojevicKa2013} to random processes \cite[Section~2.9]{SamorodnitskyTa1994}. Although it is named after Birkhoff \cite{Birkhoff1935}, there are earlier papers on this orthogonality type by Radon \cite{Radon1916} and Blaschke \cite{Blaschke1916}. Properties of isosceles orthogonality, a notion introduced by James \cite{James1945}, are determined by the geometry of bisectors. Being a decisive tool for creating Voronoi diagrams, these sets of points of equal distance from given objects have a great impact on computational geometry and discrete geometry; see, for instance, \cite{DezaSi2015} for recent research in this direction.

In order to take the next step and transfer these concepts to generalized Minkowski spaces, we introduce binary relations $\perp$ on $\RR^d$ via proximality with respect to certain convex functions. Therefore, we fix our notation and outline necessary concepts from convex analysis and convex optimization in the second section. In the classical theory of normed spaces, the presence of the following properties indicates \enquote{how much} a given orthogonality notion resembles usual Euclidean orthogonality:
\begin{enumerate}[label={(\alph*)},leftmargin=*,align=left,noitemsep]
\item{Nondegeneracy: For all $x\in\RR^d$ and $\lambda,\mu\in\RR$, $\lambda x\perp\mu x$ if and only if $\lambda\mu x=0$.}
\item{Symmetry: For all $x,y\in\RR^d$, $x\perp y$ implies $y\perp x$.}
\item{Right additivity: For all $x,y,z\in\RR^d$, $x\perp y$ and $x\perp z$ together imply $x\perp (y+z)$.}
\item{Left additivity: For all $x,y,z\in\RR^d$, $x\perp z$ and $y\perp z$ together imply $(x+y)\perp z$.}
\item{Right homogeneity: For all $x,y\in\RR^d$ and $\lambda>0$, $x\perp y$ implies $x\perp\lambda y$.}
\item{Left homogeneity: For all $x,y\in\RR^d$ and $\lambda>0$, $x\perp y$ implies $\lambda x\perp y$.}
\item{Right existence: For all $x,y\in\RR^d$, there exists a number $\alpha\in\RR$ such that $x\perp(\alpha x+y)$.}
\item{Left existence: For all $x,y\in\RR^d$, there exists a number $\alpha\in\RR$ such that $(\alpha x+y)\perp x$.}
\end{enumerate}
Checking these properties for approximate Birkhoff orthogonality and isosceles orthogonality in generalized Minkowski spaces is (sometimes more, sometimes less explicitly) the subject of \hyperref[chap:birkhoff]{Sections~\ref*{chap:birkhoff}} and \ref{chap:isosceles}. We conclude the presentation by giving some future research perspectives in the fifth and final section.

\section{Notation and preliminaries}
Throughout the paper, we are concerned with the $d$-dimensional real vector space $\RR^d$. This space is equipped with the usual topology induced by the Euclidean inner product $\skpr{\cdot}{\cdot}$ and the Euclidean norm. We shall use the notions of \emph{interior}, \emph{closure}, and \emph{boundary} accordingly. We abbreviate the \emph{linear hull} of a set $A\subset\RR^d$ by $\lin A$. The \emph{straight line} passing through $x,y\in\RR^d$ and \emph{ray} starting at $x\in\RR^d$ and passing through $y\in\RR^d$ are denoted by $\strline{x}{y}=\setcond{x+\lambda(y-x)}{\lambda\in\RR}$ and $\clray{x}{y}=\setcond{x+\lambda(y-x)}{\lambda\geq 0}$, respectively. A \emph{gauge} is a function $\gamma:\RR^d\to\RR$ which meets the following requirements:
\begin{enumerate}[label={(\alph*)},leftmargin=*,align=left,noitemsep]
\item{$\gamma(x)\geq 0$ for all $x\in\RR^d$, and $\gamma(x)=0$ implies $x=0$,}
\item{$\gamma(\lambda x)=\lambda\gamma(x)$ for all $x\in\RR^d$, $\lambda>0$,}
\item{$\gamma(x+y)\leq\gamma(x)+\gamma(y)$.}
\end{enumerate}
In particular, $\gamma$ is called \emph{rotund} if $\gamma(x+y)<2$ whenever $x,y\in\RR^d$, $x\neq y$, and $\gamma(x)=\gamma(y)=1$. It is called \emph{G{\^{a}}teaux differentiable} at $x\in\RR^d$, if the directional derivative
\begin{equation*}
\gamma^\prime(x;y)=\lim_{\lambda\downarrow 0}\frac{\gamma(x+\lambda y)-\gamma(x)}{\lambda}
\end{equation*}
is linear in $y$.  The \emph{polar function} of $\gamma$ is defined as
\begin{equation*}
\gamma^\circ:\RR^d\to\RR,\qquad\gamma^\circ(x^\ast)=\inf\setcond{\lambda>0}{\skpr{x^\ast}{x}\leq\lambda\gamma(x)\fall x\in\RR^d}.
\end{equation*}
The polar function $\gamma^\circ$ is again a gauge and satisfies the Cauchy--Schwarz inequality
\begin{equation}
\skpr{x^\ast}{x}\leq\gamma^\circ(x^\ast)\gamma(x)\label{eq:cauchy-schwarz}
\end{equation}
automatically. 

A set $B\subset\RR^d$ is convex if the line segment $[x,y]\defeq\setcond{\lambda x+(1-\lambda)y}{0\leq\lambda\leq 1}$ is contained in $B$ for all $x,y\in B$.
A convex set $B$ is \emph{rotund} if there is no line segment contained in its boundary.
For $x^\ast\in\RR^d\setminus\setn{0}$ and $\alpha\in\RR$, the hyperplane $\setcond{y\in\RR^d}{\skpr{x^\ast}{y}=\alpha}$ is a \emph{supporting hyperplane} of $B$ provided $B$ is contained in the closed half-space $\setcond{y\in\RR^d}{\skpr{x^\ast}{y}\leq\alpha}$ but not in its interior.
The set $B$ is called \emph{smooth} if, for every boundary point $x$ of $B$, there is a unique supporting hyperplane of $B$ passing through $x$.
An \emph{affine diameter} of $B$ is a line segment $[x,y]$ joining two boundary points of $B$ (that is, a \emph{chord} of $B$) which admit distinct parallel supporting hyperplanes passing through them.
The \emph{support function} of $B$ is $h_B:\RR^d\to\RR\cup\setn{+\infty,-\infty}$, $h_B(x^\ast)=\sup\setcond{\skpr{x^\ast}{x}}{x\in B}$, and the set $B^\circ\defeq\setcond{x^\ast\in\RR^d}{h_{B}(x^\ast)\leq 1}$ is the \emph{polar set} of $B$.
A set $K\subset\RR^d$ is said to be a \emph{cone} provided $\lambda K=K$ for all $\lambda>0$.

The formulas $\gamma(x)=\inf\setcond{\lambda>0}{x\in\lambda B}$ and $B=\setcond{x\in\RR^d}{\gamma(x)\leq 1}$ establish a one-to-one correspondence between gauges and compact and convex sets having the origin as interior point.
Analogously to the classical theory of normed spaces, we define the \emph{ball with radius $\lambda$ and center $x$} by $\mc{x}{\lambda}=\setcond{y\in\RR^d}{\gamma(y-x)\leq\lambda}$. In particular, $\mc{0}{1}$ is the \emph{unit ball} of the generalized Minkowski space $(\RR^d,\gamma)$.
The intimate relationship between gauges and unit balls goes beyond this: We will see in \hyperref[result:smoothness]{Theorems~\ref*{result:smoothness}} and \ref{result:rotundity} that rotund gauges have rotund unit balls, G{\^{a}}teaux differentiable gauges have smooth unit balls, and vice versa.
Polarity also behaves well in the interplay between gauges and unit balls.
Namely, the polar gauge $\gamma^\circ$ coincides with the support function $h_{\mc{0}{1}}$, and its unit ball is the polar set of $\mc{0}{1}$. Therefore, the polar gauge of a norm is again a norm (known as the dual norm). 

It is known that convex functions $f:\RR^d\to\RR$---that is, $f(\lambda x+(1-\lambda)y)\leq\lambda f(x)+(1-\lambda)f(y)$ for all $x,y\in\RR^d$ and $\lambda\in [0,1]$---are continuous and their \emph{$\eps$-subdifferential}
\begin{equation*}
\p_\eps f(x)\defeq\setcond{x^\ast\in\RR^d}{\skpr{x^\ast}{y-x}\leq f(y)-f(x)+\eps\text{ for all }y\in\RR^d}
\end{equation*}
is non-empty, compact, and convex at each point $x\in\RR^d$, see \cite[Corollary~8.30]{BauschkeCo2011} and \cite[Theorem~2.4.2(i)]{Zalinescu2002}. The $\eps$-directional derivative
\begin{equation*}
f^\prime_\eps(x;y)\defeq \inf_{\lambda> 0}\frac{f(x+\lambda y)-f(x)+\eps}{\lambda}
\end{equation*}
coincides with the support function of $\p_\eps f(x)$, that is,
\begin{equation}
f^\prime_\eps(x;y)=\sup\setcond{\skpr{x^\ast}{y}}{x^\ast\in\p_\eps f(x)}\label{eq:directional-derivative-subdifferential-sup}
\end{equation}
for all $x,y\in\RR^d$, see \cite[Theorem~2.1.14, Theorem~2.4.9]{Zalinescu2002}. Substituting $-y$ for $y$ yields 
\begin{equation}
-f^\prime_\eps(x;-y)=\inf\setcond{\skpr{x^\ast}{y}}{x^\ast\in\p_\eps f(x)}.\label{eq:directional-derivative-subdifferential-inf}
\end{equation}
By compactness of $\p_\eps f(x)$ and continuity of $\skpr{x^\ast}{\cdot}$, the supremum in \eqref{eq:directional-derivative-subdifferential-sup} and the infimum in \eqref{eq:directional-derivative-subdifferential-inf} are attained as a maximum and a minimum respectively. Thus
\begin{equation*}
\setcond{\skpr{x^\ast}{y}}{x^\ast\in\p_\eps f(x)}=[-f^\prime_\eps(x;-y),f^\prime_\eps(x;y)],
\end{equation*}
see also \cite[Proposition~I.2.5]{Cioranescu1990} and \cite[Equation~(9)]{Precupanu2013}.

We shall omit $\eps$ from the notation if it equals zero, that is $f^\prime\defeq f^\prime_0$ and $\p f\defeq \p_0 f$. Note that every \emph{subgradient} $x^\ast\in\p f(x)$ defines a supporting hyperplane of the \emph{sublevel set} $\setcond{y\in\RR^d}{f(y)\leq f(x)}$ at $x$. For the special case $f=\gamma$, this fact is the basis for the concept of \emph{Birkhoff orthogonality}, which we will discuss in \hyperref[chap:birkhoff]{Section~\ref*{chap:birkhoff}}.

\begin{Bsp}
We have
\begin{equation}
\p_\eps\gamma(x)=\begin{cases}\setcond{x^\ast\in\RR^d}{\gamma^\circ(x^\ast)\leq 1}=\mc{0}{1}^\circ,&x=0,\\\setcond{x^\ast\in\RR^d}{\skpr{x^\ast}{x}\geq\gamma(x)-\eps,\gamma^\circ(x^\ast)\leq 1},&\text{else,}\end{cases}\label{eq:subdifferential-gauge}
\end{equation}
see \cite[Theorem~2.4.2(ii)]{Zalinescu2002}.
\end{Bsp}

The first lemma serves as an $\eps$-version of Fermat's rule and can be proved analogously to its classical counterpart \cite[Theorem~16.2, Proposition~17.17]{BauschkeCo2011}. Half of it is stated, e.g., in \cite[Theorem~XI.1.1.5]{HiriartUrrutyLe1993b}.
\begin{Lem}\label{result:fermat-rule}
Let $f:\RR^d\to\RR$ be a convex function and let $x\in\RR^d$, $\eps\geq 0$. The following statements are equivalent:

\begin{enumerate}[label={(\alph*)},leftmargin=*,align=left,noitemsep]
\item{$f(x)\leq f(y)+\eps$ for all $y\in\RR^d$,\label{eps-minimizer}}
\item{$0\in\p_\eps f(x)$\label{eps-subdifferential-contains-zero}}
\item{$f^\prime_\eps(x;y)\geq 0$ for all $y\in\RR^d$.\label{eps-directional-derivative-is-nonnegative}}
\end{enumerate}
\end{Lem}
\begin{proof}
\ref{eps-minimizer}$\Longleftrightarrow$\ref{eps-subdifferential-contains-zero}: Both statements are equivalent to $0\leq f(y)-f(x)+\eps$ for all $y\in\RR^d$.

\ref{eps-subdifferential-contains-zero}$\Longrightarrow$\ref{eps-directional-derivative-is-nonnegative}: We have 
\begin{align*}
0\leq f(y)-f(x)+\eps \fall y\in\RR^d&\Longrightarrow 0\leq f(x+ty)-f(x)+\eps \fall y\in\RR^d, t>0\\
&\Longrightarrow 0\leq \frac{f(x+ty)-f(x)+\eps}{t} \fall y\in\RR^d, t>0\\
&\Longrightarrow 0\leq f^\prime_\eps(x;y),
\end{align*}

\ref{eps-directional-derivative-is-nonnegative}$\Longrightarrow$\ref{eps-subdifferential-contains-zero}: We have
\begin{align*}
0\leq f^\prime_\eps(x;y)\fall y\in\RR^d&\Longrightarrow 0\leq f^\prime_\eps(x;y-x)\fall y\in\RR^d\\
&\Longrightarrow 0\leq \frac{f(x+t(y-x))-f(x)+\eps}{t}\fall y\in\RR^d, t>0\\ 
&\Longrightarrow 0\leq f(y)-f(x)+\eps \fall y\in\RR^d.
\end{align*}
\end{proof}

For $\eps=0$, \hyperref[result:fermat-rule]{Lemma~\ref*{result:fermat-rule}} gives a characterization of minimizers of a convex function in terms of the subdifferential. The existence of a minimizer of a convex function can be guaranteed by the additional assumption of coercitivity, see \cite[Theorem~11.9]{BauschkeCo2011}. A function $f:\RR^d\to\RR$ is called \emph{coercive} if all of its sublevel sets $\setcond{x\in\RR^d}{f(x)\leq\alpha}$ are bounded.

\section{Approximate Birkhoff orthogonality}\label{chap:birkhoff}
Approximate orthogonality relations are usually defined by introducing a relaxation parameter $\eps$. For Birkhoff orthogonality in normed spaces, there is more than one approach to an approximate version. Some of them are connected to semi-inner products \cite{Chmielinski2005b,Dragomir1991,MalPaSa2017a,ChmielinskiStWo2017} (see \cite{KaramanlisPs2013,MoslehianZa2015b} for applications), but we will follow another approach \cite{HaddadiMaSh2009,HasaniMaVa2007} which works well with $\eps$-subdifferentials and $\eps$-best approximations.
\begin{Def}\label{definition:orthogonality}
The vector $x\in\RR^d$ is called \emph{$\eps$-Birkhoff orthogonal to $y\in\RR^d$} (abbreviated by $x\perp_B^\eps y$) if $\gamma(x)\leq\gamma(x+\lambda y)+\eps$ for all $\lambda\in\RR$. If $\eps=0$, we shall omit $\eps$ from the notation and simply refer to \emph{Birkhoff orthogonality}.
\end{Def}
Trivially, the nondegeneracy property and the following homogeneity property are true for approximate Birkhoff orthogonality: For every $x,y\in\RR^d$, $\lambda>0$, and $\mu\in\RR$, we have that $x\perp_B^\eps y$ implies $\lambda x\perp_B^{\lambda\eps}\mu y$. The remainder of this section is subdivided into three parts which address existence properties of $\eps$-Birkhoff orthogonality, additivity of $0$-Birkhoff orthogonality, and symmetry of $\eps$-Birkhoff orthogonality, respectively.
\subsection{Dual characterizations}\label{chap:dual-characterizations}
Since various numerical methods for solving convex optimization problems can be derived from inclusion problems like $0\in\p f(x)$, rephrasing optimality (in the sense of minimizing a certain convex function $f$) via \hyperref[result:fermat-rule]{Lemma~\ref*{result:fermat-rule}} is central to convex optimization.
In the classical convex analysis of real normed spaces $(X,\mnorm{\cdot})$, subdifferentials are subsets of the topological dual space $(X^\ast,\mnorm{\cdot}_\ast)$ whose elements are continuous linear functionals mapping $X$ to the real numbers.
Thus, Fermat's rule provides a dual characterization of optimal solutions of convex optimization problems.

Another family of set-valued operators which are of interest for giving analytical descriptions of the geometry of normed spaces is given by the so-called \emph{duality mappings} $J_\phi:X\rightrightarrows X^\ast$, see, e.g., \cite{Browder1965}, \cite[Chapters~I,~II]{Cioranescu1990}, or \cite[Section~3.7]{Zalinescu2002}.

In the following, we will show how $\eps$-Birkhoff orthogonality relates to convex optimization problems and certain proximality notions and give dual descriptions thereof. Some of our results require $\eps=0$. In such a case, their novelty compared to the existing literature lies in the usage of gauges instead of norms.
In this spirit, we start by introducing duality mappings in generalized Minkowski spaces $(\RR^d,\gamma)$, see also \cite[Section~2.4.7]{Cobzas2013} for a related discussion in the context of asymmetric moduli of rotundity and smoothness.

\begin{Def}\label{definition:duality-mapping}
Let $(\RR^d,\gamma)$ be a generalized Minkowski space, and let $\phi:[0,+\infty)\to [0,+\infty)$ be \emph{weight}, i.e., a continuous, non-decreasing, and non-negative function. The \emph{duality mapping with weight $\phi$} is the set-valued operator $J_\phi\gamma:\RR^d\rightrightarrows\RR^d$,
\begin{equation*}
J_\phi\gamma(x)=\setcond{x^\ast\in\RR^d}{\skpr{x^\ast}{x}=\gamma^\circ(x^\ast)\gamma(x),\gamma^\circ(x^\ast)=\phi(\gamma(x))}.
\end{equation*}
\end{Def}
\begin{Bem}
\begin{enumerate}[label={(\alph*)},leftmargin=*,align=left,noitemsep]
\item{Our notation alters the classical one (cf. \cite{Browder1965}) in order to emphasize the dependency on $\gamma$.}
\item{We do not distinguish between the primal space and its topological dual in our $\RR^d$ setting.}
\item{If $\gamma=\mnorm{\cdot}$ is a norm, then $\gamma^\circ=\mnorm{\cdot}_\ast$ is the dual norm. Therefore, \hyperref[definition:duality-mapping]{Definition~\ref*{definition:duality-mapping}} is an extension of the classical notion.}
\item{In the literature, the weight $\phi:[0,+\infty)\to [0,+\infty)$ is assumed to be strictly increasing and to meet the additional requirements $\phi(0)=0$ and $\lim_{t\to+\infty}\phi(t)=+\infty$. In this case, we have $J_\phi\gamma(0)=\setn{0}\neq\p\gamma(0)$ independently of $\phi$ and $\gamma$. In our definition, $J_\phi\gamma=\p\gamma$ when $\phi(t)=1$ for all $t>0$.}
\end{enumerate}
\end{Bem}
When $\phi:[0,+\infty)\to [0,+\infty)$ is the identity, Cior{\u{a}}nescu \cite[Chapters~I,~II]{Cioranescu1990} refers to $J_\phi(x)$ as the \emph{normalized duality mapping} at $x$. Peypouquet \cite[Section~1.1.2]{Peypouquet2015} uses this term for $\p\mnorm{\cdot}(x)$ at non-zero points $x\in\RR^d$. In any case, normalization is not important as duality mappings turn out to be rescalings of each other at non-zero points $x\in\RR^d$. The following result extends \cite[Theorem~I.4.4]{Cioranescu1990} to generalized Minkowski spaces.

\begin{Satz}\label{result:duality-mapping}
Let $\phi:[0,+\infty)\to [0,+\infty)$ be a weight, and let $\psi:[0,+\infty)\to [0,+\infty)$, $\psi(t)=\int_0^t\phi(s)\dd s$. Then $\psi$ is a non-decreasing convex function, $\psi\circ\gamma:\RR^d\to\RR$ is a convex function, and $J_\phi\gamma(x)=\p(\psi\circ\gamma)(x)$ for all $x\in\RR^d$.
\end{Satz}
\begin{proof}
The convexity of $\psi$ is a consequence of the fundamental theorem of calculus, see \cite[Lemma~I.4.3]{Cioranescu1990}. Furthermore, $\psi$ is differentiable and its derivative is the non-negative function $\psi^\prime=\phi$ which means that $\psi$ is non-decreasing. By \cite[Theorem~5.1]{Rockafellar1972}, $\psi\circ\gamma$ is a convex function.
The chain rule for subdifferentials \cite[Theorem~2.8.10]{Zalinescu2002} yields 
\begin{align*}
\p(\psi\circ\gamma)(x)&=\setcond{\alpha x^\ast}{\alpha\in\p\psi(\gamma(x)),x^\ast\in\p\gamma(x)}\\
&=\psi^\prime(\gamma(x)\p\gamma(x)=J_\phi\gamma(x).\qedhere
\end{align*}
\end{proof}
An easy consequence is that  $\phi_2(\gamma(x))J_{\phi_1}(x)=\phi_1(\gamma(x))J_{\phi_2}(x)$ for all $x\in\RR^d$ and for all weights $\phi_1,\phi_2:[0,+\infty)\to [0,+\infty)$, see \cite[Proposition~I.4.7(f)]{Cioranescu1990}.

Necessary and sufficient conditions for $x\perp_B^\eps(\alpha x+y)$ can be given in terms of linear functionals, see \cite[Corollary~2.2]{James1947a} and \cite[Remark~15]{Dragomir2004} for the non-relaxed version in normed spaces. Note that given $x\in\RR^d$ and $\eps\geq\gamma(x)$, we have $x\perp_B^\eps y$ for all $y\in\RR^d$. Furthermore, \eqref{eq:subdifferential-gauge} gives $\p_\eps\gamma(0)=\mc{0}{1}^\circ$ independently of $\eps$. Therefore, the restriction to $0\leq\eps<\gamma(x)$ is justified when asking for statements which link $\eps$-subdifferentials to $\eps$-Birkhoff orthogonality.
\begin{Satz}\label{result:right-birkhoff-characterization}
Let $(\RR^d,\gamma)$ be a generalized Minkowski space. Furthermore, let $x,y\in\RR^d$, $\alpha\in\RR$, $0\leq\eps<\gamma(x)$, and define $h:\RR\to\RR$, $h(\lambda)=\gamma(x+\lambda(\alpha x+y))$. The following statements are equivalent:
\begin{enumerate}[label={(\alph*)},leftmargin=*,align=left,noitemsep]
\item{$x\perp_B^\eps(\alpha x+y)$,\label{right-birkhoff-orthogonality}}
\item{$h(0)\leq h(\lambda)+\eps$ for all $\lambda\in\RR$,\label{auxiliary_function_minimum}}
\item{there exists $x^\ast\in\RR^d$ such that $\gamma^\circ(x^\ast)=1$, $\skpr{x^\ast}{x}\geq\gamma(x)-\eps$ and $\alpha=-\frac{\skpr{x^\ast}{y}}{\skpr{x^\ast}{x}}$,\label{alpha-formula}}
\item{$\gamma^\prime_\eps(x;\pm(\alpha x+y))\geq 0$.\label{directional-derivative-optimality}}
\end{enumerate}
\end{Satz}
\begin{proof}
The equivalence \ref{right-birkhoff-orthogonality}$\Longleftrightarrow$\ref{auxiliary_function_minimum} is a direct consequence of the definition. The equivalences \ref{auxiliary_function_minimum}$\Longleftrightarrow$\ref{alpha-formula}$\Longleftrightarrow$\ref{directional-derivative-optimality} follow from \hyperref[result:fermat-rule]{Lemma~\ref*{result:fermat-rule}}. To this end, we show that \ref{alpha-formula} and \ref{directional-derivative-optimality} are reformulations of the conditions $0\in\p_\eps h(0)$ and $h^\prime_\eps(0;\nu)\geq 0$ for all $\nu\in\RR$, respectively. First, \cite[Theorem~2.8.10]{Zalinescu2002} yields
\begin{align*}
\p_\eps h(\lambda)&=\bigcup_{\substack{\eps_1,\eps_2\geq 0\\\eps_1+\eps_2=\eps\\q\in\p_{\eps_2}\gamma(x+\lambda(\alpha x+y))}}\p_{\eps_1}\skpr{x^\ast}{x+\cdot(\alpha x+y)}(\lambda)\\
&=\bigcup_{\substack{\eps_1,\eps_2\geq 0\\\eps_1+\eps_2=\eps\\x^\ast\in\p_{\eps_2}\gamma(x+\lambda(\alpha x+y))}}\setn{\skpr{x^\ast}{\alpha x+y}}\\
&=\setcond{\skpr{x^\ast}{\alpha x+y}}{x^\ast\in\p_\eps\gamma(x+\lambda(\alpha x+y))},
\end{align*}
where the first equality holds because $a:\RR\to\RR$, $a(\lambda)=\skpr{x^\ast}{x+\lambda(\alpha x+y)}$ is an affine function.

Using \cite[Theorem~2.4.9]{Zalinescu2002}, we obtain
\begin{align*}
h^\prime_\eps(\lambda;\nu)&=\sup\setcond{\mu\nu}{\mu\in\p_\eps h(\lambda)}\\
&=\sup\setcond{\skpr{x^\ast}{\alpha x+y}\nu}{x^\ast\in\p_\eps\gamma(x+\lambda(\alpha x+y))}\\
&=\sup\setcond{\skpr{x^\ast}{\nu(\alpha x+y)}}{x^\ast\in\p_\eps\gamma(x+\lambda(\alpha x+y))}\\
&=\gamma^\prime_\eps(x+\lambda(\alpha x+y),\nu(\alpha x+y)).
\end{align*}

In particular,
\begin{align*}
h^\prime_\eps(0;\nu)&=\gamma^\prime_\eps(x;\nu(\alpha x+y)),\\
\p_\eps h(0)&=\setcond{\skpr{x^\ast}{\alpha x+y}}{x^\ast\in\p_\eps\gamma(x)}.
\end{align*}
Taking the positive homogeneity of the $\eps$-directional derivative in the second variable into account, it suffices to consider $\nu=\pm 1$ for checking whether $h^\prime(0;\nu)\geq 0$ for all $\nu\in\RR$. 

Moreover, we have 
\begin{align*}
&0\in\p_\eps h(0)\\
&\Longleftrightarrow \text{ there exists } x^\ast\in\p_\eps\gamma(x) \text{ such that } \skpr{x^\ast}{\alpha x+y}=0\\
&\Longleftrightarrow \text{ there exists } x^\ast\in\RR^d \text{ such that } \gamma^\circ(x^\ast)\leq 1,\skpr{x^\ast}{x}\geq\gamma(x)-\eps, \alpha=-\frac{\skpr{x^\ast}{y}}{\skpr{x^\ast}{x}}\\
&\Longleftrightarrow \text{ there exists } x^\ast\in\RR^d \text{ such that } \gamma^\circ(x^\ast)=1,\skpr{x^\ast}{x}\geq\gamma(x)-\eps, \alpha=-\frac{\skpr{x^\ast}{y}}{\skpr{x^\ast}{x}}.
\end{align*}
This completes the proof.
\end{proof}

For $\eps=0$, the characterization of Birkhoff orthogonality in terms of directional derivatives given in \hyperref[result:right-birkhoff-characterization]{Theorem~\ref*{result:right-birkhoff-characterization}}\ref{right-birkhoff-orthogonality}$\Longleftrightarrow$\ref{directional-derivative-optimality} can be rewritten in a form resembling \cite[Theorem~3.2]{James1947a}.
\begin{Kor}\label{result:right-birkhoff-directional-derivative}
Let $(\RR^d,\gamma)$ be a generalized Minkowski space, $x,y\in\RR^d$, and $\alpha\in\RR$. Then $x\perp_B(\alpha x+y)$ implies
\begin{equation}
-\gamma^\prime(x;-y)\leq-\alpha\gamma(x)\leq\gamma^\prime(x;y).\label{eq:right-birkhoff-directional-derivative-characterization}
\end{equation}
\end{Kor}
\begin{proof}
\hyperref[result:right-birkhoff-characterization]{Theorem~\ref*{result:right-birkhoff-characterization}}, we have $x\perp_B(\alpha x+y)$ if and only if $\gamma^\prime(x;\pm(\alpha x+y))\geq 0$.

Using \hyperref[eq:subdifferential-gauge]{Equation~(\ref*{eq:subdifferential-gauge})}, we obtain $\skpr{x^\ast}{x}=\gamma(x)$ for all $x^\ast\in\p\gamma(x)$ and thus
\begin{align*}
\gamma^\prime(x;\alpha x+\mu y)&=\max\setcond{\alpha\skpr{x^\ast}{x}+\mu\skpr{x^\ast}{y}}{x^\ast\in\p\gamma(x)}\\
&=\max\setcond{\alpha\gamma(x)+\mu\skpr{x^\ast}{y}}{x^\ast\in\p\gamma(x)}\\
&=\alpha\gamma(x)+\mu\max\setcond{\skpr{x^\ast}{y}}{x^\ast\in\p\gamma(x)}\\
&=\alpha\gamma(x)+\mu\gamma^\prime(x;y)
\end{align*}
for all $x,y\in\RR^d$, $\alpha\in\RR$, and $\mu\geq 0$.
\end{proof}

As a corollary of \hyperref[result:duality-mapping]{Theorems~\ref*{result:duality-mapping}} and \ref{result:right-birkhoff-characterization}, we obtain the following result, see \cite[Equation~(18)]{Precupanu2013} and \cite[Proposition~I.4.10]{Cioranescu1990} for a special case of its part \ref{annihilating-weighted-subgradient}.
\begin{Kor}\label{result:birkhoff-orthogonality-annihilating-subgradient}
Let $\phi:[0,+\infty)\to [0,+\infty)$ be a weight. For all $x,y\in\RR^d$, $0\leq\eps<\gamma(x)$, we have
\begin{enumerate}[label={(\alph*)},leftmargin=*,align=left,noitemsep]
\item{$x\perp_B^\eps y$ if and only if there exists $x^\ast\in\p_\eps\gamma(x)$ such that $\skpr{x^\ast}{y}=0$;\label{annihilating-eps-subgradient}}
\item{$x\perp_B y$ if and only if there exists $x^\ast\in J_\phi\gamma(x)$ such that $\skpr{x^\ast}{y}=0$.\label{annihilating-weighted-subgradient}}
\end{enumerate}
\end{Kor}

As a consequence of \hyperref[result:birkhoff-orthogonality-annihilating-subgradient]{Corollary~\ref*{result:birkhoff-orthogonality-annihilating-subgradient}}\ref{annihilating-weighted-subgradient},
\begin{align}
\setcond{\alpha\in\RR}{x\perp(\alpha x+y)}&=\setcond{\frac{\skpr{x^\ast}{y}}{\skpr{x^\ast}{x}}}{x^\ast\in J_\phi(x)}\nonumber\\
&=\setcond{-\frac{\skpr{x^\ast}{y}}{\gamma(x)\phi(\gamma(x))}}{x^\ast\in J_\phi(x)}\label{eq:right-parameters}
\end{align}
is a non-empty compact interval provided $x\neq 0$, see \cite[Corollary~7, Remark~8]{Precupanu2013} and \cite[Corollary~11, Remark~16]{Dragomir2004}. 

As a reformulation of \hyperref[eq:right-parameters]{Equation~(\ref*{eq:right-parameters})}, duality maps can be written in terms of Birkhoff orthogonality analogously to \cite[Theorem~5]{Precupanu2013} on which, in turn, the proof is patterned.
\begin{Satz}\label{result:duality-mapping-via-birkhoff-orthogonality}
Let $(X,\gamma)$ be a generalized Minkowski space and let $\phi:[0,+\infty)\to [0,+\infty)$ be a weight. For $x\neq 0$,
\begin{equation}
J_\phi\gamma(x)=\setcond{x^\ast\in\RR^d}{x\perp_B\lr{y-\frac{\skpr{x^\ast}{y}}{\gamma(x)\phi(\gamma(x))}x}\fall y\in\RR^d}
\end{equation}
\end{Satz}
\begin{proof}
Let $x^\ast\in J_\phi\gamma(x)$. Then $\skpr{x^\ast}{y-\frac{\skpr{x^\ast}{y}}{\gamma(x)\phi(\gamma(x))}x}=0$. Taking \hyperref[result:birkhoff-orthogonality-annihilating-subgradient]{Corollary~\ref*{result:birkhoff-orthogonality-annihilating-subgradient}} into account, we have $x\perp_B\lr{y-\frac{\skpr{x^\ast}{y}}{\gamma(x)\phi(\gamma(x))}x}$.

Conversely, let $x^\ast\in\RR^d$ such that $x\perp_B\lr{y-\frac{\skpr{x^\ast}{y}}{\gamma(x)\phi(\gamma(x))}}x$ for all $y\in\RR^d$.
By \hyperref[result:birkhoff-orthogonality-annihilating-subgradient]{Corollary~\ref*{result:birkhoff-orthogonality-annihilating-subgradient}}, there exists vectors $u_y^\ast\in J_\phi\gamma(x)$, i.e., $\skprs{u_y^\ast}{x}=\gamma^\circ(u_y^\ast)\gamma(x)$, $\gamma^\circ(u_y^\ast)=\phi(\gamma(x))$, such that
\begin{equation*}
\skpr{u_y^\ast}{y-\frac{\skpr{x^\ast}{y}}{\gamma(x)\phi(\gamma(x))}x}=0.
\end{equation*}
Thus $\skprs{u_y^\ast}{y}=\skpr{x^\ast}{y}$ for all $y\in\RR^d$.
In particular, for $y=x$ we obtain $\skpr{x^\ast}{x}=\skprs{u_x^\ast}{x}=\phi(\gamma(x))\gamma(x)$, i.e., $\gamma^\circ(x^\ast)\geq\phi(\gamma(x))$.
On the other hand, for all $y\in\RR^d$, $\skpr{x^\ast}{y}=\skprs{u_y^\ast}{y}\leq\gamma^\circ(u_y^\ast)\gamma(y)=\phi(\gamma(x))\gamma(y)$, so $\gamma^\circ(x^\ast)\leq\phi(\gamma(x))$. Summarizing, $\gamma^\circ(x^\ast)=\phi(\gamma(x))$ and the proof is complete.
\end{proof}

Extending \cite[Corollary~2.2, Lemma~3.1]{James1947a} to generalized Minkowski spaces, there is an analogous statement about the numbers $\alpha\in\RR$ for which $x\perp_B^\eps(\alpha x+y)$.
\begin{Prop}\label{result:right-birkhoff-existence}
Let $x,y\in\RR^d$ and $0\leq\eps<\gamma(x)$. The set of numbers $\alpha\in\RR$ for which $x\perp_B^\eps (\alpha x+y)$ is a non-empty compact interval. In particular, if $x\perp_B^\eps(\alpha x+y)$, then $\abs{\alpha}\leq\max\setn{\frac{\gamma(y)}{\gamma(x)-\eps},\frac{\gamma(-y)}{\gamma(x)-\eps}}$.
\end{Prop}
\begin{proof}
Take $x^\ast\in\p_\eps\gamma(x)\neq\emptyset$, i.e., $\gamma^\circ(x^\ast)\leq 1$ and $\skpr{x^\ast}{x}\geq\gamma(x)-\eps>0$. If $\skpr{x^\ast}{h}=0$, then $x\perp_B^\eps h$, see \hyperref[result:birkhoff-orthogonality-annihilating-subgradient]{Corollary~\ref*{result:birkhoff-orthogonality-annihilating-subgradient}}. Assume that, for all $\alpha\in\RR$, the vector $x$ is not $\eps$-Birkhoff orthogonal to $\alpha x+y$. In particular, the line $\setcond{\alpha x+y}{\alpha\in\RR}$ does not intersect the hyperplane $\setcond{h\in\RR^d}{\skpr{x^\ast}{h}=0}$. Consequently, $\skpr{x^\ast}{x}=0$, a contradiction.

By \hyperref[definition:orthogonality]{Definition~\ref*{definition:orthogonality}}, $x\perp_B^\eps(\alpha x+y)$ if and only if $\gamma(x)-\eps\leq\gamma(x+\lambda(\alpha x+y))$ for all $\lambda\in\RR$. In particular, for $\alpha\neq 0$ and $\lambda=-\frac{1}{\alpha}$, we obtain $\gamma(x)-\eps\leq\gamma\lr{-\frac{1}{\alpha}y}$. If $\alpha>0$, then $\gamma(x)-\eps\leq\frac{1}{\alpha}\gamma(-y)$. In case $\alpha<0$, we have $\gamma(x)-\eps\leq -\frac{1}{\alpha}\gamma(y)$. This yields $\abs{\alpha}\leq\max\setn{\frac{\gamma(y)}{\gamma(x)-\eps},\frac{\gamma(-y)}{\gamma(x)-\eps}}$ for $\alpha\neq 0$, which holds trivially for $\alpha=0$.

By \hyperref[result:birkhoff-orthogonality-annihilating-subgradient]{Corollary~\ref*{result:birkhoff-orthogonality-annihilating-subgradient}}, we have
\begin{align*}
\setcond{\alpha\in\RR}{x\perp_B^\eps(\alpha x+y)}&=\setcond{-\frac{\skpr{x^\ast}{y}}{\skpr{x^\ast}{x}}}{x^\ast\in\p_\eps\gamma(x)}\\
&=\setcond{-\frac{\skpr{x^\ast}{y}}{\skpr{x^\ast}{x}}}{x^\ast\in\p_\eps\gamma(x),\skpr{x^\ast}{x}=\gamma(x)-\eps}\\
&=\setcond{-\frac{\skpr{x^\ast}{y}}{\gamma(x)-\eps}}{x^\ast\in\p_\eps\gamma(x),\skpr{x^\ast}{x}=\gamma(x)-\eps}.
\end{align*}
Since $x^\ast\mapsto -\frac{\skpr{x^\ast}{y}}{\gamma(x)-\eps}$ is a linear function and $\setcond{x^\ast\in\p_\eps\gamma(x)}{\skpr{x^\ast}{x}=\gamma(x)-\eps}$ is a compact convex set, $\setcond{\alpha\in\RR}{x\perp_B^\eps(\alpha x+y)}$ is compact convex set, too. 
\end{proof}

A feature of gauges is their possible asymmetry. More precisely, we may think of $\gamma(x-y)$ as the distance from $y$ to $x$ which need not coincide with $\gamma(y-x)$. In this sense, the definition of Birkhoff orthogonality states that $x\perp_B^\eps y$ if and only if $x$ is \enquote{approximately closest} to $0$ among the points of the form $x+\lambda y$ with $\lambda\in\RR$. This kind of proximality is comprised in the notion of $\eps$-best approximation, which is naturally accompanied by the notion of $\eps$-best co-approximation. In normed spaces, the former has been introduced by Buck \cite{Buck1965}, the latter by Hasani, Mazaheri, and Vaezpour \cite{HasaniMaVa2007}, despite the fact that best co-approximations (for $\eps=0$) have already been investigated by Franchetti and Furi \cite{FranchettiFu1972}. 

\begin{Def}
Let $K$ be a non-empty closed convex subset of a generalized Minkowski space $(\RR^d,\gamma)$. A point $x\in K$ is called an \emph{$\eps$-best approximation of $y\in\RR^d$ in $K$} if $\gamma(x-y)\leq\gamma(z-y)+\eps$ for all $z\in K$.
A point $x\in K$ is called an \emph{$\eps$-best co-approximation of $y\in\RR^d$ in $K$} if $\gamma(x-z)\leq\gamma(y-z)+\eps$ for all $z\in K$. The sets of $\eps$-best approximations and $\eps$-best co-approximations of $y$ in $K$ shall be denoted by $P_K^\eps(y)$ and $R_K^\eps(y)$, respectively.
\end{Def}
The set $R_K^\eps(y)$ can be readily checked for closedness and convexity as
\begin{equation*}
R_K^\eps(y)=K\cap\lr{\bigcap_{z\in K}\mc{z}{\gamma(y-z)+\eps}}
\end{equation*}
is the intersection of closed convex sets.
\begin{Bsp}
For any $y\in\RR^d$, the set of $\eps$-best approximations of $y$ in $\RR^d$ and the set of $\eps$-best co-approximations of $y$ in $\RR^d$ coincide with $\mc{y}{\eps}$.
\end{Bsp}

The next result links $\eps$-Birkhoff orthogonality and $\eps$-best approximations in linear subspaces to $\eps$-subdifferentials. The proof follows the lines of \cite[Theorem~6.12]{Singer1970} which is the corresponding result for normed spaces, see also \cite[Lemma~1.1]{HaddadiMaSh2009}, \cite[Theorem~2.3]{HasaniMaVa2007}, and \cite[Theorem~2.5.1]{Cobzas2013}.
\begin{Prop}\label{result:birkhoff-best-approximation}
Let $U\subset\RR^d$ be a non-trivial linear subspace, $x,y\in\RR^d$, $y\notin U$, $x\in U$ and $\eps\geq 0$. Then $P_U^\eps(y)$ is non-empty, closed, and convex. Moreover, the following statements are equivalent:
\begin{enumerate}[label={(\alph*)},leftmargin=*,align=left,noitemsep]
\item{The point $x$ is an $\eps$-best approximation of $y$ in $U$.\label{best-approximation}}
\item{We have $(x-y)\perp_B^\eps z$ for all $z\in U$.\label{birkhoff-difference}}
\item{There exists $x^\ast\in\RR^d$ such that $\gamma^\circ(x^\ast)=1$, $\skpr{x^\ast}{u}=0$ for all $u\in U$, and $x^\ast\in\p_\eps\gamma(x-y)$.\label{dual-characterization}}
\end{enumerate}
\end{Prop}

We prepare the proof of \hyperref[result:birkhoff-best-approximation]{Proposition~\ref*{result:birkhoff-best-approximation}} by giving a geometrical description of the $\eps$-subdifferential of a gauge whose non-relaxed analog for normed spaces is presented in \cite[Theorem~2.1]{James1947a}.
\begin{Lem}\label{result:birkhoff-orthogonality-cauchy-schwarz-inequality}
Let $x,x^\ast\in\RR^d$, $\gamma^\circ(x^\ast)=1$, and $0\leq\eps<\gamma(x)$. The following statements are equivalent:
\begin{enumerate}[label={(\alph*)},leftmargin=*,align=left,noitemsep]
\item{$x^\ast\in\p_\eps\gamma(x)$;\label{eps-subgradient}}
\item{$\skpr{x^\ast}{x}\geq\gamma(x)-\eps$;\label{cauchy-schwarz-eps-equality}}
\item{$\skpr{x^\ast}{x}>0$, $x\perp_B^\eps h$ for all $h\in\RR^d$ with $\skpr{x^\ast}{h}=0$;\label{birkhoff-orthogonality-eps-hyperplane}}
\item{$\sup\setcond{\skpr{x^\ast}{z}}{z\in\RR^d,\gamma(z)\leq\gamma(x)-\eps}\leq\skpr{x^\ast}{x}$.\label{linear-functional-eps-maximum}}
\end{enumerate}
\end{Lem}
\begin{proof}
\ref{eps-subgradient}$\Longleftrightarrow$\ref{cauchy-schwarz-eps-equality}: See \hyperref[eq:subdifferential-gauge]{Equation~(\ref*{eq:subdifferential-gauge})}.

\ref{eps-subgradient}$\Longrightarrow$\ref{birkhoff-orthogonality-eps-hyperplane}: We have
\begin{equation}
\gamma(y)+\eps\geq\gamma(x)+\skpr{x^\ast}{y-x}\label{eq:eps-subgradient}
\end{equation}
for all $y\in\RR^d$. If $y$ admits a representation $y=x+h$ with $h\in\RR^d$, $\skpr{x^\ast}{h}=0$, then \eqref{eq:eps-subgradient} becomes
\begin{equation*}
\gamma(x+h)+\eps\geq\gamma(x)+\skpr{x^\ast}{h}=\gamma(x),
\end{equation*}
that is, $x\perp_B^\eps h$. By \ref{cauchy-schwarz-eps-equality}, we also have $\skpr{x^\ast}{x}\geq\gamma(x)-\eps>0$.

\ref{birkhoff-orthogonality-eps-hyperplane}$\Longrightarrow$\ref{cauchy-schwarz-eps-equality}: Set $\mu\defeq\frac{\skpr{x^\ast}{x}}{\gamma(x)-\eps}>0$. For every $y\in\RR^d$ there are a number $\lambda\in\RR$ and a point $h\in\RR^d$ such that $\skpr{x^\ast}{h}=0$ and $y=\lambda x+h$. If $\lambda<0$, then $\skpr{x^\ast}{y}=\lambda\skpr{x^\ast}{x}<0<\mu\gamma(y)$. If $\lambda\geq 0$, we obtain
\begin{equation*}
\skpr{x^\ast}{y}=\lambda\skpr{x^\ast}{x}=\lambda\mu(\gamma(x)-\eps)\leq\mu\gamma(\lambda x+h)=\mu\gamma(y).
\end{equation*}
Thus $1=\gamma^\circ(x^\ast)\leq\mu$, which is equivalent to $\skpr{x^\ast}{x}\geq\gamma(x)-\eps$.

\ref{linear-functional-eps-maximum}$\Longleftrightarrow$\ref{cauchy-schwarz-eps-equality}: This follows from the identity
\begin{equation*}
\sup_{z\in\mc{0}{\gamma(x)-\eps}}\skpr{x^\ast}{z}=\gamma^\circ(x^\ast)(\gamma(x)-\eps)=\gamma(x)-\eps,
\end{equation*}
which is a consequence of $\gamma^\circ=h_{\mc{0}{1}}$.
\end{proof}

In case of $\eps=0$, the previous lemma can be slightly improved to
\begin{Lem}\label{result:birkhoff-orthogonality-cauchy-schwarz-inequality-zero}
Let $x,x^\ast\in\RR^d$, $x^\ast\neq 0$. The following statements are equivalent:
\begin{enumerate}[label={(\alph*)},leftmargin=*,align=left,noitemsep]
\item{$\frac{x^\ast}{\gamma^\circ(x^\ast)}\in\p\gamma(x)$;}
\item{$\skpr{x^\ast}{x}=\gamma^\circ(x^\ast)\gamma(x)$;}
\item{$\skpr{x^\ast}{x}\geq 0$, $x\perp_B h$ for all $h\in\RR^d$ with $\skpr{x^\ast}{h}=0$;}
\item{$\skpr{x^\ast}{\cdot}$ attains its maximum on $\mc{0}{\gamma(x)}$ at $x$.}
\end{enumerate}
\end{Lem}

The nontriviality of \hyperref[result:birkhoff-orthogonality-cauchy-schwarz-inequality]{Lemma~\ref*{result:birkhoff-orthogonality-cauchy-schwarz-inequality}} is a consequence of the nonemptiness of the $\eps$-sub\-differential of $\gamma$, see \cite[Theorem~2.4.9]{Zalinescu2002}.
\begin{Lem}[{\cite[Theorem~2.2]{James1947a}}]\label{result:birkhoff-orthogonality-eps-hyperplane-nontriviality}
Let $x\in\RR^d$. Then there exists a vector $x^\ast\in\RR^d\setminus\setn{0}$ such that $x\perp_B^\eps h$ whenever $\skpr{x^\ast}{h}=0$.
\end{Lem}

\begin{myproof}
Apply \cite[Theorem~11.9]{BauschkeCo2011} to the convex, hence continuous, and coercive function $\gamma(\cdot-y):U\to\RR$ and conclude that $P_U^0(y)$ is non-empty. Therefore, the set $P_U^\eps(y)$ is also non-empty and, being a sublevel set of $\gamma(\cdot-y):U\to\RR$, compact and convex.

\ref{best-approximation}$\Longleftrightarrow$\ref{birkhoff-difference}: Use $x+\lambda U=U$.

\ref{best-approximation}$\Longrightarrow$\ref{dual-characterization}: We have $\varrho\defeq\inf_{z\in U}\gamma(z-y)>0$.
Using the Hahn--Banach theorem, there exists a hyperplane $H$ which separates $U$ and $\mc{y}{\varrho}$. As $U\cap\mc{y}{\varrho}=P_U^0(y)$ is non-empty, $H$ contains $U$ and is therefore a linear subspace itself.
Thus there exists $x^\ast\in\RR^d$ such that $H=\setcond{z\in\RR^d}{\skpr{x^\ast}{z}=0}$.
Moreover, $H$ is a supporting hyperplane of $\mc{y}{\varrho}$.
Hence we may choose $x^\ast$ such that $\mc{y}{\varrho}\subset H^-\defeq\setcond{z\in\RR^d}{\skpr{x^\ast}{z}\leq 0}$.
We conclude that $\mc{0}{\varrho}\subset H^--y$. Choosing $x_0\in P_U^0(y)$ and applying \hyperref[result:birkhoff-orthogonality-cauchy-schwarz-inequality-zero]{Lemma~\ref*{result:birkhoff-orthogonality-cauchy-schwarz-inequality-zero}}, the linear functional $\skpr{x^\ast}{\cdot}$ restricted to $\mc{0}{\varrho}$ attains its minimum value $\skpr{x^\ast}{-y}$ at $x_0-y$, so $\frac{x^\ast}{\gamma^\circ(x^\ast)}\in\p\gamma(x_0-y)$. This means that we may choose $x^\ast$ such that $\gamma^\circ(x^\ast)=1$ and $\skpr{x^\ast}{x_0-y}=\gamma(x_0-y)$.

Since $x\in P_U^\eps(y)$, we have
\begin{equation*}
\gamma(x-y)\leq\varrho+\eps=\gamma(x_0-y)+\eps=\skpr{x^\ast}{x_0-y}+\eps=\skpr{x^\ast}{x-y}+\eps.
\end{equation*}

\ref{dual-characterization}$\Longrightarrow$\ref{best-approximation}: By virtue of \eqref{eq:cauchy-schwarz} and \eqref{eq:subdifferential-gauge}, we have
\begin{equation*}
\gamma(x-y)\leq\skpr{x^\ast}{x-y}+\eps=\skpr{x^\ast}{u-y}+\eps\leq\gamma(u-y)+\eps
\end{equation*}
for all $u\in U$.
\end{myproof}

In contrast to that, there are sufficient conditions for $\eps$-best co-approximations of points in linear subspaces in terms of $\eps$-Birkhoff orthogonality and $\eps$-subdifferentials which need not be necessary ones.
Closely related results in finite-dimensional normed spaces are, for instance, \cite[p.~1046, (1)]{FranchettiFu1972}, \cite[Proposition~2.1]{PapiniSi1979}, and \cite[Theorems~2.3,~2.6,~2.10]{HasaniMaVa2007}.
\begin{Prop}\label{result:birkhoff-best-coapproximation}
Let $U\subset\RR^d$ be a non-trivial linear subspace, $x,y\in\RR^d$, $y\notin U$, $x\in U$ and $\eps\geq 0$. Then each of the following three equivalent statements
\begin{enumerate}[label={(\alph*)},leftmargin=*,align=left,noitemsep,series=coapprox]
\item{We have $z\perp_B^\eps(y-x)$ for all $z\in U$.\label{birkhoff-codifference}}
\item{For all $z\in U$, $x$ is an $\eps$-best approximation of $z$ in $\strline{x}{y}$.\label{reverse-best-approximation}}
\item{For $z\in U$, there exists $x^\ast\in\RR^d$ such that $\gamma^\circ(x^\ast)=1$, $\skpr{x^\ast}{u}=0$ for all $u\in U$, and $x^\ast\in\p_\eps\gamma(y-x)$.\label{dual-cocharacterization}}
\end{enumerate}
implies that
\begin{enumerate}[label={(\alph*)},leftmargin=*,align=left,noitemsep,coapprox]
\item{The point $x$ is an $\eps$-best co-approximation of $y$ in $U$.\label{best-coapproximation}}
\end{enumerate}
\end{Prop}

\begin{proof}
\ref{birkhoff-codifference}$\Longleftrightarrow$\ref{reverse-best-approximation}: As $\tilde{z}$ traverses $U$, the point $z=x-\tilde{z}$ does the same, and vice versa. So
\begin{equation*}
\gamma(\tilde{z})\leq\gamma(\tilde{z}+\lambda(y-x))+\eps
\end{equation*}
for all $\tilde{z}\in U$ and all $\lambda\in\RR$ if and only if
\begin{equation*}
\gamma(x-z)\leq\gamma(x+\lambda(y-x)-z)+\eps
\end{equation*}
for all $z\in U$, $\lambda\in\RR$. (Note that $x+\lambda(y-x)$ is an arbitrary point of $\strline{x}{y}$.)

\ref{reverse-best-approximation}$\Longleftrightarrow$\ref{dual-cocharacterization}: See \hyperref[result:birkhoff-best-approximation]{Proposition~\ref*{result:birkhoff-best-approximation}}\ref{best-approximation}$\Longleftrightarrow$\ref{dual-characterization}.

\ref{reverse-best-approximation}$\Longrightarrow$\ref{best-coapproximation}: In
\begin{equation*}
\gamma(x-z)\leq\gamma(w-z)+\eps \text{ for all }z\in U\text{ and }w\in\strline{x}{y},
\end{equation*}
we choose $w=y$ to obtain
\begin{equation*}
\gamma(x-z)\leq\gamma(y-z)+\eps \text{ for all }z\in U.
\end{equation*}
\end{proof}
If $\gamma$ is a norm and $\eps=0$, the implication \ref{best-coapproximation}$\Longrightarrow$\ref{birkhoff-codifference} is true as well. However, we cannot expect this implication to be valid for gauges in general, even for $\eps=0$. For instance, take $X=\RR^2$, $\gamma:\RR^2\to\RR$, $\gamma(x_1,x_2)=\max\setn{-x_2,x_2-x_1,x_2+x_1}$, $U=\setcond{(x_1,0)}{x_1\in\RR}$, and $y=(0,1)$. Then the unit ball $\mc{0}{1}$ is the triangle with vertices $(0,1)$, $(-2,-1)$, $(2,-1)$, and the set of $0$-best co-approximations of $y$ in $U$ is
\begin{equation*}
R_U^0(y)=U\cap\lr{\bigcap_{z\in U}\mc{z}{\gamma(y-z)}}=[(-1,0),(1,0)].
\end{equation*}
Now take $x=(0,0)$, $z=(1,0)$, and $w=(0,-0.5)$. Then $\gamma(w-z)=0.5<1=\gamma(x-z)$, so $x$ is not a $0$-best approximation of $z$ in $\strline{x}{y}$. Furthermore, for $\lambda=-0.5$, $\gamma(z)=1>0.5=\gamma(z+\lambda(y-x))$, so $z\centernot\perp_B(y-x)$.

On the other hand, item \ref{best-coapproximation} from \hyperref[result:birkhoff-best-coapproximation]{Proposition~\ref*{result:birkhoff-best-coapproximation}} implies that $\gamma(z)\leq\gamma(z+\alpha(y-x))+\eps$ for all $z\in U$ and $\alpha\in [0,1]$, see again \cite[Theorem~2.3]{HasaniMaVa2007} for the analogous statement in normed spaces.

As a corollary of \hyperref[result:birkhoff-best-approximation]{Proposition~\ref*{result:birkhoff-best-approximation}} for one-dimensional subspaces, we obtain the following result, see also \cite[Theorem~2.3]{James1947a}.
\begin{Kor}\label{result:left-birkhoff-characterization}
Let $x,y\in\RR^d$, $\eps\geq 0$, and $\alpha\in\RR$. The following statements are equivalent:
\begin{enumerate}[label={(\alph*)},leftmargin=*,align=left,noitemsep]
\item{The point $\alpha x+y$ is an $\eps$-best approximation of $0$ in $y+\lin\setn{x}$.}
\item{The point $\alpha x$ is an $\eps$-best approximation of $-y$ in $\lin\setn{x}$.}
\item{We have $(\alpha x+y)\perp_B^\eps x$.}
\item{There exists $x^\ast\in\RR^d$ such that $\gamma^\circ(x^\ast)=1$, $\skpr{x^\ast}{x}=0$ and $\skpr{x^\ast}{y}\geq\gamma(\alpha x+y)-\eps$.}
\end{enumerate}
Moreover, the set of numbers $\alpha\in\RR$ such that $(\alpha x+y)\perp_B^\eps x$ is a compact interval provided $x\neq 0$.
\end{Kor}

Emulating key properties of usual inner products, \emph{semi-inner products} enable Hilbert-like arguments in arbitrary Banach spaces. Prominent examples include the \emph{superior} and \emph{inferior semi-inner product} associated with a given norm $\mnorm{\cdot}$, which are in fact directional derivates of the convex function $\frac{1}{2}\mnorm{\cdot}^2$. New approaches to classical concepts have been developed, not only to optimization problems like the Fermat--Torricelli problem \cite{ComanescuDrKi2013} and the best approximation problem \cite{Dragomir2001} but also to geometric concepts like orthogonality \cite[Chapters~8--11]{Dragomir2004}. In particular, several results connecting semi-inner products to Birkhoff orthogonality have been derived. We close this subsection by demonstrating how Birkhoff orthogonality in generalized Minkowski spaces can be characterized in terms of natural analogs of the superior and inferior semi-inner products. To this end, consider the functions $g:\RR^d\to\RR$, $g(x)=\frac{1}{2}\gamma(x)^2$, and $(\cdot,\cdot)_s, (\cdot,\cdot)_i:\RR^d\times\RR^d\to\RR$ defined by
\begin{align*}
(y,x)_s&=\lim_{\lambda\downarrow 0}\frac{\gamma(x+\lambda y)^2-\gamma(x)^2}{2\lambda}=g^\prime(x;y),\\
(y,x)_i&=\lim_{\lambda\uparrow 0}\frac{\gamma(x+\lambda y)^2-\gamma(x)^2}{2\lambda}=-g^\prime(x;-y)=-(-y,x)_s.
\end{align*}
Note that $(\cdot,\cdot)_s$, $(\cdot,\cdot)_i$ need not be semi-inner products in the sense of \cite[Definition~6]{Dragomir2004} since $(x,y)_p^2\leq (x,x)_p(y,y)_p$ for $p\in\setn{s,i}$ may be invalidated by the asymmetry of $\gamma$. In normed spaces, this estimate is checked in \cite[Proposition~6]{Dragomir2004}. The proof uses the reverse triangle inequality which is not valid for gauges. However, we can show an upper bound for $(\cdot,\cdot)_s$:
\begin{align}
(y,x)_s&=\lim_{\lambda\downarrow 0}\frac{\gamma(x+\lambda y)^2-\gamma(x)^2}{2\lambda}\nonumber\\
&=\lim_{\lambda\downarrow 0}\lr{\frac{\gamma(x+\lambda y)+\gamma(x)}{2}\cdot\frac{\gamma(x+\lambda y)-\gamma(x)}{\lambda}}\nonumber\\
&=\lim_{\lambda\downarrow 0}\frac{\gamma(x+\lambda y)+\gamma(x)}{2}\lim_{\lambda\downarrow 0}\frac{\gamma(x+\lambda y)-\gamma(x)}{\lambda}\nonumber\\
&=\gamma(x)\lim_{\lambda\downarrow 0}\frac{\gamma(x+\lambda y)-\gamma(x)}{\lambda}\label{eq:chain-rule}\\
&\leq\gamma(x)\lim_{\lambda\downarrow 0}\frac{\gamma(\lambda y)}{\lambda}\nonumber\\
&\leq\gamma(x)\gamma(y),\nonumber
\end{align}
where \eqref{eq:chain-rule} can be written as $(y,x)_s=\gamma(x)\gamma^\prime(x;y)$, which is basically the chain rule for directional derivatives \cite[Proposition~3.6]{Shapiro1990} applied to the function $g$. Similarly, a lower bound for the function $(\cdot,\cdot)_i$ is $(y,x)_i=-\gamma(x)\gamma^\prime(x;-y)\geq -\gamma(x)\gamma(-y)$.

Following the lines of \cite[Proposition~5]{Dragomir2004}, we can also check that
\begin{align*}
(x,x)_s&=\lim_{\lambda\downarrow 0}\frac{\gamma(x+\lambda x)^2-\gamma(x)^2}{2\lambda}=\lim_{\lambda\downarrow 0}\frac{(1+\lambda)^2\gamma(x)^2-\gamma(x)^2}{2\lambda}\\
&=\gamma(x)^2\lim_{\lambda\downarrow 0}\frac{(1+\lambda)^2-1}{2\lambda}
=\gamma(x)^2
\end{align*}
and
\begin{align*}
(x,x)_i&=\lim_{\lambda\uparrow 0}\frac{\gamma(x+\lambda x)^2-\gamma(x)^2}{2\lambda}
=\lim_{\lambda\uparrow 0,\lambda>-1}\frac{(1+\lambda)^2\gamma(x)^2-\gamma(x)^2}{2\lambda}\\
&=\gamma(x)^2\lim_{\lambda\uparrow 0,\lambda>-1}\frac{(1+\lambda)^2-1}{2\lambda}
=\gamma(x)^2
\end{align*}
for $x\in\RR^d$.

For all $x,y\in\RR^d$, $\alpha\in\RR$, and $\mu\geq 0$, a generalization of \cite[Theorem~16]{Dragomir2004} can be established by using the computation in the proof of \hyperref[result:right-birkhoff-directional-derivative]{Corollary~\ref*{result:right-birkhoff-directional-derivative}} and the chain rule for directional derivatives:
\begin{align*}
g^\prime(x;\alpha x+\mu y)&=\gamma(x)\gamma^\prime(x;\alpha x+\mu y)\\
&=\gamma\lr{\alpha\gamma(x)+\mu\gamma^\prime(x;y)}\\
&=\alpha\gamma(x)^2+\mu\gamma(x)\gamma^\prime(x;y).
\end{align*}

In the classical theory in normed spaces, computations like these provide the basis for proving characterizations of Birkhoff orthogonality in terms of superior and inferior semi-inner products. In our context, we may define a weight $\phi:[0,+\infty)\to [0,+\infty)$ via $\phi(t)=t$, yielding $J_\phi\gamma=\p g=\gamma(x)\p\gamma(x)$.
In combination with $(y,x)_s=g^\prime(x;y)=\gamma(x)\gamma^\prime(x;y)$, results on Birkhoff orthogonality in $(\RR^d,\gamma)$ in terms of $(\cdot,\cdot)_s$ and $(\cdot,\cdot)_i$ can therefore be derived using the above theory by suitably multiplying by $\gamma(x)$.
For instance, the analog of \cite[Corollary~12]{Dragomir2004} in generalized Minkowski spaces is the equivalence of the statements
\begin{enumerate}[label={(\alph*)},leftmargin=*,align=left,noitemsep]
\item{$x\perp_B(\alpha x+y)$,}
\item{$(y,x)_i\leq-\alpha\gamma(x)^2\leq (y,x)_s$.}
\end{enumerate}
\begin{proof}
In \eqref{eq:right-birkhoff-directional-derivative-characterization}, multiply by $\gamma(x)$.
\end{proof}
For $\alpha=0$, this yields the equivalence of $x\perp_B y$ and $(y,x)_i\leq 0\leq (y,x)_s$, see \cite[Theorem~50]{Dragomir2004}, \cite[p.~54]{Hetzelt1985}, or \cite[Equation~(17)]{Precupanu2013} for the corresponding result in normed spaces. 

Finally, setting $\phi:[0,+\infty)\to [0,+\infty)$, $\phi(t)=t$ in \hyperref[eq:right-parameters]{Equation~(\ref*{eq:right-parameters})} yields $x\perp_B\lr{y-\frac{(y,x)_s}{\gamma(x)^2}x}$ for $x\neq 0$, cf. \cite[Equation~(19)]{Precupanu2013}.

\subsection{Smoothness and rotundity}\label{chap:rotundity-smoothness}
As in the case of normed spaces, Birkhoff orthogonality in generalized Minkowski spaces can be used to characterize rotundity and smoothness of the unit ball. The results for the general setting are the subject of this section. The first theorem connects smoothness and Birkhoff orthogonality; the corresponding results for normed spaces are \cite[Theorems~4.2,~5.1]{James1947a}.
\begin{Satz}\label{result:smoothness}
Let $(\RR^d,\gamma)$ be a generalized Minkowski space of dimension $d\geq 2$. The following statements are equivalent:
\begin{enumerate}[label={(\alph*)},leftmargin=*,align=left,noitemsep]
\item{The unit ball $\mc{0}{1}$ is smooth.\label{unit-ball-smooth}}
\item{The gauge $\gamma$ is G{\^{a}}teaux differentiable on $\RR^d\setminus\setn{0}$.\label{gauge-gateaux-differentiable}}
\item{If $x,y,z\in\RR^d$, $x\perp_B y$ and $x\perp_B z$, then $x\perp_B(y+z)$.\label{birkhoff-right-additivity}}
\item{For every $x,y\in\RR^d$, $x\neq 0$, there exists a unique number $\alpha\in\RR$ such that $x\perp_B(\alpha x+y)$.\label{birkhoff-right-uniqueness}}
\item{For all $x\in\RR^d$ with $\gamma(x)=1$, there exists a unique vector $x^\ast\in\RR^d$, $\gamma^\circ(x^\ast)=1$ such that $\skpr{x^\ast}{x}=1$.\label{unique-norming-functional}}
\end{enumerate}
In this case, $x^\ast$ is the G{\^{a}}teaux derivative of $\gamma$ at $x$ (items~\ref{unique-norming-functional} and~\ref{gauge-gateaux-differentiable}), $\gamma^\prime(x;y)=-\alpha\gamma(x)$ (item~\ref{birkhoff-right-uniqueness}), and the unique hyperplane of $\mc{0}{1}$ at one of its boundary points $x$ consists of all points $y$ such that $x\perp_B (y-x)$ (item~\ref{unit-ball-smooth}).
\end{Satz}
\begin{proof}
\ref{unique-norming-functional}$\Longleftrightarrow$\ref{unit-ball-smooth}: See \hyperref[result:birkhoff-orthogonality-cauchy-schwarz-inequality]{Lemma~\ref*{result:birkhoff-orthogonality-cauchy-schwarz-inequality}}.

\ref{unit-ball-smooth}$\Longrightarrow$\ref{birkhoff-right-additivity}: The unique supporting hyperplane of $\mc{0}{\gamma(x)}$ passing through $x$ has the form
\begin{equation*}
H\defeq\setcond{h\in\RR^d}{\skpr{x^\ast}{h}=\skpr{x^\ast}{x}},
\end{equation*}
where $x^\ast$ is uniquely determined up to a constant factor. Using \hyperref[result:birkhoff-orthogonality-cauchy-schwarz-inequality]{Lemma~\ref*{result:birkhoff-orthogonality-cauchy-schwarz-inequality}}\ref{birkhoff-orthogonality-eps-hyperplane}$\Longleftrightarrow$\ref{linear-functional-eps-maximum}, we conclude that $x\perp_B h$ if and only if $h\in H-x$. That is, if $x\perp_B y$ and $x\perp_B z$, then $y,z\in H-x$. Since $H-x$ is a linear subspace of $\RR^d$, we also have $y+z\in H-x$ and thus $x\perp_B(y+z)$.

\ref{birkhoff-right-additivity}$\Longrightarrow$\ref{birkhoff-right-uniqueness}: Assume that there exist numbers $\alpha,\beta\in\RR$ such that $x\perp_B(\alpha x+y)$ and $x\perp_B(\beta x+y)$. Due to the homogeneity of $\perp_B$ and \ref{birkhoff-right-additivity}, we obtain $x\perp_B(\alpha x-\beta x)$, which implies $\alpha=\beta$.

$\lnot$\ref{unit-ball-smooth}$\Longrightarrow\lnot$\ref{birkhoff-right-uniqueness}: Let $x^\ast_1,x^\ast_2\in\RR^d$ be vectors such that the hyperplanes
\begin{equation*}
H_i\defeq\setcond{h\in\RR^d}{\skpr{x^\ast_i}{h}=\skprs{x^\ast_i}{x}},\qquad i\in\setn{1,2},
\end{equation*}
are supporting hyperplanes of $\mc{0}{\gamma(x)}$ at $x$. Set $\alpha_i\defeq -\frac{\skpr{x^\ast_i}{x}}{\skpr{x^\ast}{y}}$. Then $x\perp_B(\alpha_ix+y)$ for $i\in\setn{1,2}$, and $\alpha_1\neq\alpha_2$ if the intersection points of the line $\setcond{\alpha x+y}{\alpha\in\RR}$ with $H_1$ and $H_2$ do not coincide.

\ref{gauge-gateaux-differentiable}$\Longleftrightarrow$\ref{unit-ball-smooth}: If $\gamma$ has a G{\^{a}}teaux derivate $x^\ast$ at a point $x$ with $\gamma(x)=1$, then $\p\gamma(x)=\setn{x^\ast}$. In particular, $\gamma^\circ(x^\ast)=1$, and \hyperref[result:birkhoff-orthogonality-cauchy-schwarz-inequality]{Lemma~\ref*{result:birkhoff-orthogonality-cauchy-schwarz-inequality}}\ref{eps-subgradient}$\Longleftrightarrow$\ref{linear-functional-eps-maximum} yields the uniqueness of supporting hyperplanes of $\mc{0}{1}$ at $x$. Conversely, if there is a unique supporting hyperplane of $\mc{0}{1}$ at $x$, then \hyperref[result:birkhoff-orthogonality-cauchy-schwarz-inequality]{Lemma~\ref*{result:birkhoff-orthogonality-cauchy-schwarz-inequality}}\ref{eps-subgradient}$\Longleftrightarrow$\ref{linear-functional-eps-maximum} implies that $\p\gamma(x)$ is a singleton. Hence $\gamma$ is G{\^{a}}teaux differentiable at $x$, see \cite[Corollary~2.4.10]{Zalinescu2002}.
\end{proof}
Alternatively, the implication \ref{birkhoff-right-additivity}$\Longleftrightarrow$\ref{unit-ball-smooth} in \hyperref[result:smoothness]{Theorem~\ref*{result:smoothness}} can be proved as follows. Let $x$ be a boundary point of $\mc{0}{1}$. The set $\setcond{z\in\RR^d}{x\perp_B^\eps z}$ is a union of hyperplanes passing through the origin, hence a cone. If \ref{birkhoff-right-additivity} is true, then this cone is convex. But a convex cone which is a union of hyperplanes is either a single hyperplane or $\RR^d$, the latter one contradicting $x\neq 0$.

The second theorem is a characterization of rotund gauges in terms of Birkhoff orthogonality; it generalizes \cite[Theorems~4.3,~5.2]{James1947a}.
\begin{Satz}\label{result:rotundity}
Let $(\RR^d,\gamma)$ be a generalized Minkowski space of dimension $d\geq 2$. The following statements are equivalent:
\begin{enumerate}[label={(\alph*)},leftmargin=*,align=left,noitemsep]
\item{The unit ball $\mc{0}{1}$ is rotund.\label{unit-ball-rotund}}
\item{The gauge $\gamma$ is rotund.\label{unit-ball-chords-midpoint-interior}}
\item{For every $x,y\in\RR^d$, $x\neq 0$, there exists a unique number $\alpha\in\RR$ such that $(\alpha x+y)\perp_B x$.\label{birkhoff-left-uniqueness}}
\item{For all $x^\ast\in\RR^d$, the linear functional $\skpr{x^\ast}{\cdot}$ has at most one maximum on $\mc{0}{1}$.\label{linear-functional-maximum-uniqueness}}
\end{enumerate}
\end{Satz}

\begin{proof}
\ref{unit-ball-rotund}$\Longleftrightarrow$\ref{linear-functional-maximum-uniqueness}: The family of supporting hyperplanes of $\mc{0}{1}$ coincides with the family of hyperplanes
\begin{equation*}
H=\setcond{y\in\RR^d}{\skpr{x^\ast}{y}=h_{\mc{0}{1}}(x^\ast)},
\end{equation*}
where $x^\ast\in\RR^d$. The set of maximizers of $\skpr{x^\ast}{\cdot}$ on $\mc{0}{1}$ is then $H\cap\mc{0}{1}$, which is a subset of the boundary of $\mc{0}{1}$.

\ref{unit-ball-rotund}$\Longleftrightarrow$\ref{unit-ball-chords-midpoint-interior}: See \cite[Lemma~3.5]{JahnKuMaRi2015}.

\ref{unit-ball-rotund}$\Longleftrightarrow$\ref{birkhoff-left-uniqueness}: Fix $x,y\in\RR^d$, $x\neq 0$. Let $\mu\defeq\min\setcond{\gamma(\lambda x+y)}{\lambda\in\RR}$. Then $\mc{0}{\mu}\cap\setcond{\lambda x+y}{\lambda\in\RR}$ is the set of points $\lambda x+y$ for which $(\lambda x+y)\perp_B x$. This set is compact and convex. Thus, if it is not a singleton, it is a segment, and there is a ball which contains a straight line segment in its boundary. Conversely, if $\mc{0}{1}$ is not rotund, choose a segment $[y,z]$ in the boundary of $\mc{0}{1}$ and set $x=z-y$. Then the line $\setcond{\lambda x+y}{\lambda\in\RR}$ does not meet the interior of $\mc{0}{1}$ and, hence, $(\lambda x+y)\perp_B x$ for $\lambda\in\setn{0,1}$.
\end{proof}
\hyperref[result:smoothness]{Theorem~\ref*{result:smoothness}} shows that right additivity characterizes smoothness in all dimensions. However, left additivity does not play the same role for rotundity. 
\begin{Satz}
Let $(\RR^d,\gamma)$ be a generalized Minkowski space of dimension $d\geq 2$. If Birkhoff orthogonality is left additive, then $\gamma$ is a norm.
\end{Satz}
\begin{proof}
Let $d=2$. If $\gamma$ is not a norm, then there is an affine diameter not passing through the origin, see \cite[Section~4.1]{Soltan2005}. In other words, there are vectors $x,y,z\in\RR^d$ such that $x$ and $y$ are linearly independent, $\gamma(x)=\gamma(y)=1$, $z\neq 0$, $x\perp_B z$, $y\perp_B z$, and $x-y\notin\lin\setn{z}$.
Thus, there are numbers $\lambda,\mu\in\RR$ with $0<\lambda<1$ and $\lambda x+(1-\lambda) y=\mu z$. 
But left additivity and homogeneity imply $(\alpha x+\beta y)\perp_B z$ for all numbers $\alpha,\beta>0$. 
Therefore $\mu z\perp_B z$, which implies $z=0$.

If $d\geq 3$ and Birkhoff orthogonality is left additive, then it is left additive in each two-dimensional subspace of $\RR^d$. This implies, using the first part of the proof, that the restriction of $\gamma$ to any two-dimensional subspace of $\RR^d$ is a norm on that subspace. Hence, $\gamma$ itself is a norm.
\end{proof}
Therefore, left additivity for gauges reduces to the case of norms which can be found in \cite[pp.~561,~562]{James1947b}.
\begin{Kor}
Let $(\RR^d,\gamma)$ be a generalized Minkowski space of dimension $d\geq 2$. Assume that Birkhoff orthogonality is left additive. If $d=2$, then $\gamma$ is rotund. Else $\gamma$ is a norm induced by an inner product.
\end{Kor}

\subsection{Orthogonality reversion and symmetry}\label{chap:symmetry}
Norms whose Birkhoff orthogonality relations coincide were studied in \cite{Schoepf1996} and \cite[Theorem~10]{Precupanu2013}, and the two-dimensional special case is implicitly stated, e.g., in \cite[p.~165f.]{Eggleston1965} and \cite[p.~90]{Thompson1996}. The analogous investigation for gauges on $\RR^2$ was done in \cite[4A]{Schaeffer1976}. As the proof of \cite[Theorem~10]{Precupanu2013} is not based on the symmetry property of norms but on general facts like the maximal monotonicity of subdifferentials of convex functions, we may translate the result to our setting and omit the proof.

\begin{Satz}
Let $\gamma_1,\gamma_2:\RR^d\to\RR$ two gauges whose Birkhoff orthogonality relations shall be denoted $\perp_{B,1}$ and $\perp_{B,2}$. Furthermore, let $\phi:[0,+\infty)\to [0,+\infty)$, $\phi(t)=t$. The following statements are equivalent:
\begin{enumerate}[label={(\alph*)},leftmargin=*,align=left,noitemsep]
\item{There exists a number $\kappa>0$ such that $\gamma_1(x)=\kappa\gamma_2(x)$ for all $x\in\RR^d$.\label{proportional-gauges}}
\item{For all $x,y\in\RR^d$, we have $x\perp_{B,1}y$ if and only if $x\perp_{B,2}y$.\label{birkhoff-relations}}
\item{For all $x\in\RR^d\setminus\setn{0}$, $\frac{1}{\gamma_1(x)}J_\phi\gamma_1(x)=\frac{1}{\gamma_2(x)}J_\phi\gamma_2(x)$.\label{duality-map}}
\end{enumerate}
\end{Satz}

The identification of pairs of norms whose Birkhoff orthogonality relations are inverses of each other yields the notion of \emph{antinorm} in two-dimensional spaces, see \cite[p.~867]{Busemann1947}, \cite[Proposition~3.1]{Hetzelt1985}, and \cite{MartiniSw2006}. For normed spaces of dimension at least three, this class reduces to pairs of norms whose unit balls are homothetic ellipsoids, see \cite[Theorem~3.2]{HorvathLaSp2015}. Closely related are norms whose Birkhoff orthogonality relation is symmetric.
In the two-dimensional case, these norms are named after Johann Radon \cite{Radon1916}. From \cite[Theorem~3.2]{HorvathLaSp2015} it follows that in higher dimensions, symmetry of Birkhoff orthogonality characterizes Euclidean spaces. However, this result is much older and goes back to Blaschke \cite{Blaschke1916}.

In the present section, we prove that there are no asymmetric analogs of the antinorm and of Radon norms. 
\begin{Satz}
Let $\gamma_1,\gamma_2:\RR^d\to\RR$ two gauges whose $\eps$-Birkhoff orthogonality relations shall be denoted $\perp_{B,1}^\eps$ and $\perp_{B,2}^\eps$. Assume that for all $x,y\in\RR^d$ such that $x\perp_{B,1}^\eps y$ if and only if $y\perp_{B,2}^\eps x$. Then $\gamma_1$ is a norm and $\eps=0$.
\end{Satz}
\begin{proof}
Let $x,y\in\RR^d\setminus\setn{0}$. Due to homogeneity and the assumption, we have
\begin{align}
x\perp_{B,1}^\eps y&\Longleftrightarrow y\perp_{B,2}^\eps x\nonumber\\
&\Longleftrightarrow y\perp_{B,2}^\eps\frac{-x}{\gamma_1(-x)}\nonumber\\
&\Longleftrightarrow\frac{-x}{\gamma(-x)}\perp_{B,1}^\eps y.\label{eq:birkhoff-reversion}
\end{align}
Case $\eps=0$: If $\gamma_1(x)=1$, then $x$ and $\frac{-x}{\gamma_1(-x)}$ are the endpoints of a chord of the convex body $B\defeq\setcond{x\in\RR^d}{\gamma_1(x)\leq 1}$ which passes through the origin $0$. From \eqref{eq:birkhoff-reversion} and the separation theorems it follows that there is a pair of parallel supporting hyperplanes of $B$ at the endpoints of every such chord. In other words, every chord passing through $0$ is an affine diameter of $B$. Since $0$ is an interior point of $B$, the claim follows by taking \cite[Section~4.1]{Soltan2005} into account.

Case $\eps>0$: Fix $x\in\RR^d$ such that $\gamma_1(x)>\eps$. Then there exists $y\in\RR^d$ such that $x\perp_{B,1}^\eps y$ and, without loss of generality, $\gamma_1(y)<\eps$. But then $\gamma_1(y)<\eps\leq\gamma_1(y+\lambda z)+\eps$, so $y\perp_{B,1}^\eps z$ for all $z\in\RR^d$. By assumption, $z\perp_{B,2}^\eps y$ for all $z\in\RR^d$, that is,
\begin{equation}
\gamma_2(z)\leq\gamma_2(z+\lambda y)+\eps\label{eq:eps-birkhoff-reversion}
\end{equation}
for all $z\in\RR^d$ and $\lambda\in\RR$. In particular, if we pick $n\in\NN$ large enough such that $n\gamma(y)>\eps$ and set $z=ny$ and $\lambda=-n$, then \eqref{eq:eps-birkhoff-reversion} becomes $n\gamma(y)\leq\eps$, a contradiction.
\end{proof}

\begin{Bem}
Birkhoff orthogonality in two-dimensional generalized Minkowski spaces can be \enquote{partially reversed} in the following sense which is patterned on the case of normed spaces, see again \cite[p.~867]{Busemann1947}. Denote by $\varrho:\RR^2\to\RR^2$ the counterclockwise rotation by $90\degr$ about the origin. Let $c:[0,2\piup)\to\RR^2$ be an injective parametrization of $S\defeq\setcond{x\in\RR^2}{\gamma(x)=1}$. Then, for all $t$, the directional derivative $c^\prime(t;1)$ exists (see \cite[Proposition~3.6]{Shapiro1990}), $c(t)\perp_B c^\prime(t;1)$ in $(\RR^2,\gamma)$, and $c^\prime(t;1)\perp_B c(t)$ in $(\RR^2,\gamma^\circ\circ\varrho)$.
\end{Bem}
\begin{Kor}
Let $(\RR^d,\gamma)$ be a generalized Minkowski space and let $\eps\geq 0$. If $\eps$-Birkhoff orthogonality is a symmetric relation, then $\eps=0$ and $\gamma$ is a norm.
\end{Kor}

\section{Isosceles orthogonality}\label{chap:isosceles}
In Euclidean elementary geometry, triangles with reflection symmetry are isosceles, and if the lengths of the diagonals of a parallelogram coincide, the parallelogram is actually a rectangle. Formally, the orthogonality of vectors $x$ and $y$ in Euclidean space is equivalent to the equality of the lengths of the vectors $x+y$ and $x-y$. Serving as a definition in normed spaces, this yields the concept of \emph{isosceles orthogonality}, which in general is different from Birkhoff orthogonality.
\begin{Def}
Let $(\RR^d,\gamma)$ be a generalized Minkowski space. We say that the point $y\in\RR^d$ is \emph{isosceles orthogonal to $x\in\RR^d$} (abbreviated by $y\perp_I x$) if $\gamma(y+x)=\gamma(y-x)$.
\end{Def}
\subsection{Symmetry and directional convexity}
In \hyperref[chap:symmetry]{Section~\ref*{chap:symmetry}}, we proved that if $\eps$-Birkhoff orthogonality relations of to generalized Minkowski spaces $(\RR^d,\gamma_1)$ and $(\RR^d,\gamma_2)$ are inverses of each other, then $\gamma_1$ and $\gamma_2$ are norms (and $\eps=0$). The corresponding orthogonality reversion result for isosceles orthogonality is as follows.
\begin{Satz}\label{result:isosceles-reversion}
Let $\gamma_1,\gamma_2:\RR^d\to\RR$ two gauges whose Birkhoff orthogonality relations shall be denoted $\perp_{I,1}$ and $\perp_{I,2}$. Assume that $x\perp_{I,1} y$ if and only if $y\perp_{I,2} x$, then $\gamma_2$ is a norm and $\gamma_1=\gamma_2$.
\end{Satz}
\begin{proof}
For all $y\in\RR^d$, we have $y\perp_{I,1} 0$. By assumption, $0\perp_{I,2} y$ for all $y\in\RR^d$, i.e., $\gamma_2(0+y)=\gamma_2(0-y)$ for all $y\in\RR^d$. Thus $\gamma_2$ is a norm. Since isosceles orthogonality is a symmetric relation in normed spaces, we have $\gamma_1=\gamma_2$.
\end{proof}
As a direct consequence of In \hyperref[result:isosceles-reversion]{Theorem~\ref*{result:isosceles-reversion}}, we obtain the following statement.
\begin{Kor}
Let $(\RR^d,\gamma)$ be a generalized Minkowski space. If isosceles orthogonality is a symmetric relation, then $\gamma$ is a norm.
\end{Kor}

The set of points which are isosceles orthogonal to a given point $x\in\RR^d$ is also known as the \emph{bisector} $\bisec(-x,x)$ of $x$ and $-x$, see \cite[Definition~2.1.0.1]{Ma2000}. The intersection of this bisector and every line parallel to $\lin\setn{x}$ is non-empty, a fact which is stated for normed spaces, e.g., in \cite[Theorem~4.4]{James1945}.
\begin{Lem}\label{result:bisector-directional-convexity}
Let $x$ be a non-zero vector of a generalized Minkowski space $(\RR^d,\gamma)$. Then, for all $y\in\RR^d$, the set of numbers $\alpha\in\RR$ with $(\alpha x+y)\perp_I x$ is a non-empty, closed, bounded, and convex, i.e., a compact interval.
\end{Lem}
\begin{proof}
For fixed $y\in\RR^d$, consider the function $f:\RR\to\RR$ given by
\begin{equation*}
f(\alpha)=\gamma(\alpha x+y+x)-\gamma(\alpha x+y-x).
\end{equation*}
Note that $f_{=0}=\setcond{\alpha\in\RR}{(\alpha x+y)\perp_I x}$. Closedness of this set is due to continuity of $f$. For $\alpha>0$, we have
\begin{equation*}
\gamma((\alpha+\lambda)x+y))-\gamma(\alpha x+y)=\gamma\lr{\alpha x+\frac{\alpha}{\alpha+\lambda}y}-\gamma(\alpha x+y)+\lambda\gamma\lr{x+\frac{1}{\alpha+\lambda}y}
\end{equation*}
provided $\alpha+\lambda>0$. Using the subadditivity of $\gamma$, we obtain
\begin{equation*}
0\leq\abs{\gamma\lr{\alpha x+\frac{\alpha}{\alpha+\lambda}y}-\gamma(\alpha x+y)}\leq\max\setn{\gamma\lr{\frac{\lambda}{\alpha+\lambda}y},\gamma\lr{-\frac{\lambda}{\alpha+\lambda}y}},
\end{equation*}
yielding
\begin{equation*}
\lim_{\alpha\to+\infty}\lr{\gamma\lr{\alpha x+\frac{\alpha}{\alpha+\lambda}y}-\gamma(\alpha x+y)}=0.
\end{equation*}
It follows that 
\begin{equation*}
\lim_{\alpha\to+\infty}\lr{\gamma((\alpha+\lambda)x+y))-\gamma(\alpha x+y)}=\lim_{\alpha\to+\infty}\lambda\gamma\lr{x+\frac{1}{\alpha+\lambda}y}=\lambda\gamma(x).
\end{equation*}
Using this equation, we have
\begin{align}
&\norel\lim_{\alpha\to+\infty}(\gamma((\alpha x+y)+x)-\gamma((\alpha x+y)-x))\notag\\
&=\lim_{\alpha\to+\infty}(\gamma((\alpha+1)x+y)-\gamma((\alpha-1)x+y))\notag\\
&=\lim_{\alpha\to+\infty}(\gamma((\alpha+2)x+y)-\gamma(\alpha x+y))\notag\\
&=2\gamma(x)>0\label{eq:right-limit}
\end{align}
and
\begin{align}
&\norel\lim_{\alpha\to-\infty}(\gamma((\alpha x+y)+x)-\gamma((\alpha x+y)-x))\notag\\
&=\lim_{\alpha\to+\infty}(\gamma((-\alpha+1)x+y)-\gamma((-\alpha-1)x+y))\notag\\
&=\lim_{\alpha\to+\infty}(\gamma((\alpha-2)(-x)+y)-\gamma(\alpha (-x)+y))\notag\\
&=-2\gamma(-x)<0.\label{eq:left-limit}
\end{align}
Using the intermediate value theorem, the continuity of $f$ yields the existence of a zero of $f$. Moreover, \eqref{eq:right-limit} and \eqref{eq:left-limit} imply that the set of zeros of $f$ is bounded. 

Now fix $y\in\RR^d$ and $\alpha>0$. We show that $\gamma(y+x)-\gamma(y-x)\leq\gamma(\alpha x+y+x)-\gamma(\alpha x+y-x)$. If $y$ is a multiple of $x$, the claim is easily seen. Else the points $y$, $-x$, $x$, and $\alpha x+y$ are (in this cyclic order) the vertices of a convex quadrangle. Now \cite[Lemma~4.4]{JahnMaRi2016} gives
\begin{equation*}
\gamma(y+x)+\gamma(\alpha x+y-x)\leq\gamma(y-x)+\gamma(\alpha x+y-x)
\end{equation*}
or, equivalently,
\begin{equation*}
\gamma(y+x)-\gamma(y-x)\leq\gamma(\alpha x+y+x)-\gamma(\alpha x+y-x).
\end{equation*}
Hence $f$ is increasing, and its sublevel sets are intervals. This yields the convexity part of the claim.
\end{proof}
In particular, \hyperref[result:bisector-directional-convexity]{Lemma~\ref*{result:bisector-directional-convexity}} shows that the intersection of $\bisec(-x,x)$ and every translate of $\strline{-x}{x}$ is either a singleton or a line segment. Next we will identify subsets of $\bisec(-x,x)$ which are unions of line segments parallel to $\strline{-x}{x}$. For $x,x^\ast\in\RR^d$ with $\gamma^\circ(x^\ast)$, we define the set
\begin{align*}
C(x,x^\ast)\defeq x+\setcond{y\in\RR^d}{\skpr{x^\ast}{y}=\gamma(y)},
\end{align*}
which is the translation by $x$ of the union of rays from the origin through the exposed face $\setcond{y\in\RR^d}{\skpr{x^\ast}{y}=1}\cap\mc{0}{1}$ of the unit ball $\mc{0}{1}$.
\begin{Prop}\label{prop:bisector-cone}
Given a non-zero vector $x$ in a generalized Minkowski space $(\RR^d,\gamma)$, we have $C(-x,x^\ast)\cap C(x,x^\ast)\subset\bisec(-x,x)$ whenever $x^\ast\in\RR^d$ and $\skpr{x^\ast}{x}=0$.
\end{Prop}
\begin{proof}
Consider
\begin{align*}
&z\in C(-x,x^\ast)\cap C(x,x^\ast)\\
&\Longleftrightarrow\begin{cases}\skpr{x^\ast}{z+x}=\gamma(z+x),\\\skpr{x^\ast}{z-x}=\gamma(z-x)\end{cases}\\
&\Longrightarrow\gamma(z+x)=\skpr{x^\ast}{z+x}+\skpr{x^\ast}{-2x}=\skpr{x^\ast}{z-x}=\gamma(z-x)\\
&\Longrightarrow z\in\bisec(-x,x).
\end{align*}
\end{proof}
The cases in which the straight line $y+\strline{-x}{x}$ intersects the bisector $\bisec(-x,x)$ in at most one point are specified in \hyperref[result:unique-bisector-point]{Theorem~\ref*{result:unique-bisector-point}} below. The following corollary of the triangle inequality serves as an auxiliary result, see \cite[Lemma~5]{MartiniSwWe2001} for the special case of normed spaces.
\begin{Lem}\label{result:unimodality}
Let $y,z\in (\RR^d,\gamma)$, $\lambda\in (0,1)$, and $w=\lambda y+(1-\lambda) z$. Then, for $x\in\RR^d$, 
\begin{equation*}
\gamma(w-x)\leq\max\setn{\gamma(y-x),\gamma(z-x)}
\end{equation*}
with equality if and only if $\gamma(w-x)=\gamma(y-x)=\gamma(z-x)$. In the case of equality, $\gamma(w-x)=\min\setcond{\gamma(v-x)}{v\in\strline{y}{z}}$ and $\gamma(w-x)=\gamma(v-x)$ for all $v\in[y,z]$.
\end{Lem}
\begin{proof}
We have
\begin{align}
\gamma(w-x)&=\gamma(\lambda y+(1-\lambda)z)-x)\notag\\
&=\gamma(\lambda(y-x)+(1-\lambda)(z-x))\notag\\
&\leq\lambda\gamma(y-x)+(1-\lambda)\gamma(z-x)\label{eq:unimodality_1}\\
&\leq\max\setn{\gamma(y-x),\gamma(z-x)}.\label{eq:unimodality_2}
\end{align}
If \eqref{eq:unimodality_2} holds with equality, then $\gamma(y-x)=\gamma(z-x)$, and if \eqref{eq:unimodality_1} holds with equality as well, then these numbers are equal to $\gamma(w-x)$. In other words, $y$, $w$, and $z$ are three collinear points on $\ms{x}{\gamma(w-x)}$. Hence $[y,z]\subset\ms{x}{\gamma(w-x)}$ or, equivalently, $\gamma(w-x)=\gamma(v-x)$ for all $v\in[y,z]$. Let $p\in\strline{y}{z}$ be such that $z=\mu y+(1-\mu)p$ for some $\mu\in (0,1)$. Applying the chain of inequalities above to $y$, $z$, and $p$, we obtain 
\begin{equation}
\gamma(z-x)\leq\max\setn{\gamma(y-x),\gamma(p-x)}.\label{eq:unimodality_3}
\end{equation}
Suppose $\gamma(p-x)<\gamma(y-x)$. Then \eqref{eq:unimodality_3} holds with equality, i.e., $\gamma(p-x)=\gamma(y-x)$. This is a contradiction. Thus $\gamma(p-x)\geq\gamma(y-x)$, which shows that $\gamma(w-x)=\min\setcond{\gamma(v-x)}{v\in\strline{y}{z}}$.
\end{proof}
\begin{Bem}
Applying the above \hyperref[result:unimodality]{Lemma~\ref*{result:unimodality}} to the generalized Minkowski space $(\RR^d,\gamma^\vee)$ where $\gamma^\vee(x)\defeq\gamma(-x)$, we get the same statements with reversed arguments, e.g., $\gamma(x-w)\leq\max\setn{\gamma(x-y),\gamma(x-z)}$ for all $x,y,z\in\RR^d$ and $w\in [y,z]$.
\end{Bem}
Given a non-zero vector $x$ in a generalized Minkowski space $(\RR^d,\gamma)$ with unit ball $B=\mc{0}{1}$, we determine now for which vectors $y\in\RR^d$ the intersection of the straight line $y+\strline{-x}{x}$ and the bisector $\bisec(-x,x)$ is a singleton and for which $y$ this intersection contains at least two points.
If $y$ and $x$ are linearly dependent, this intersection consists solely of the metric midpoint of $-x$ and $x$. Else the result depends on the shape of the part of the unit sphere $\ms{0}{1}$ lying in the two-dimensional \emph{half-flat} $\strline{-x}{x}+\clray{0}{y}$.
In particular, the ratio of the length of the segments $[-x,x]$ and the length of the maximal segment contained in $\ms{0}{1}\cap(\strline{-x}{x}+\clray{0}{y})$ being parallel to $\strline{-x}{x}$ is important for the formulation of \hyperref[result:unique-bisector-point]{Theorem~\ref*{result:unique-bisector-point}}.
Since we are taking ratios of lengths of parallel line segments, the result will not depend on whether we choose $\gamma(z_1-z_2)$, $\gamma(z_2-z_1)$, or $\gamma_{B-B}(z_1-z_2)$ to be the (possibly oriented) length of the segment $[z_1,z_2]$ as long as we do it consistently.
To show this, let $w_1,w_2,z_1,z_2\in\RR^d$ be such that $w_1\neq w_2$, $z_1\neq z_2$. Assume that there exist a number $\lambda>0$ such that $z_1-z_2=\lambda(w_1-w_2)$ or, equivalently, $z_2-z_1=\lambda(w_2-w_1)$. It follows that
\begin{equation*}
\frac{\gamma_{B-B}(z_1-z_2)}{\gamma_{B-B}(w_1-w_2)}=\frac{\gamma_B(z_1-z_2)}{\gamma_B(w_1-w_2)}=\frac{\gamma_B(z_2-z_1)}{\gamma_B(w_2-w_1)}=\lambda.
\end{equation*}
Due to this, set $\hfl(x,y)\defeq\strline{-x}{x}+\clray{0}{y}$ and
\begin{equation*}
M_y(x)\defeq\sup\setcond{\gamma_B(t-s)}{\begin{matrix}[s,t]\subset\ms{0}{1}\cap\hfl(x,y),\\\ex\lambda>0:t-s=\lambda x\end{matrix}}
\end{equation*}
for $x,y\in\RR^d$, $x\neq 0$. Since the number $M_y(x)$ only depends on $\hfl(x,y)$, the following generalization of \cite[Theorem~2.6]{JiLiWu2011} is essentially a two-dimensional description of the bisector $\bisec(-x,x)$.

\begin{Satz}\label{result:unique-bisector-point}
Let $x$ and $y$ be non-zero vectors of a generalized Minkowski space $(\RR^d,\gamma)$. If $M_y(x)\leq\frac{2\gamma(x)}{\gamma(y)}$, then there exists a unique real number $\alpha$ such that $(y+\alpha x)\perp_I x$.
\end{Satz}
\begin{proof}
The existence of at least one number $\alpha\in\RR$ with $y+\alpha x\in\bisec(-x,x)$ follows from \hyperref[result:bisector-directional-convexity]{Lemma~\ref*{result:bisector-directional-convexity}}. Suppose that there are two numbers $\alpha_1$, $\alpha_2\in\RR$ such that $\alpha_1<\alpha_2$ and $\setn{y+\alpha_1x,y+\alpha_2x}\subset\bisec(-x,x)$. Let $f:\RR\to\RR$, $f(\lambda)=\gamma(y+\lambda x)$. We have
\begin{align}
f(\alpha_1+1)&=\gamma(y+\alpha_1x+x)=\gamma(y+\alpha_1x-x)=f(\alpha_1+1)\label{eq:left-bisector-point}\\
f(\alpha_2+1)&=\gamma(y+\alpha_2x+x)=\gamma(y+\alpha_2x-x)=f(\alpha_2+1)\label{eq:right-bisector-point}
\end{align}
Since $f$ is a convex function, equations \eqref{eq:left-bisector-point} and \eqref{eq:right-bisector-point} imply that $f$ is constant on $[\alpha_1-1,\alpha_2+1]$. By \hyperref[result:unimodality]{Lemma~\ref*{result:unimodality}}, this constant equals $\eta\defeq\min\setcond{f(\lambda)}{\lambda\in\RR}$. Therefore, the line segment $[y+(\alpha_1-1)x,y+(\alpha_2+1)x]$ is contained in $\ms{0}{\eta}$, and we have
\begin{align*}
M_x(y)\geq\frac{1}{\eta}(\alpha_2-\alpha_1+2)\gamma(x)\geq\frac{1}{\gamma(y)}(\alpha_2-\alpha_1+2)\gamma(x)>2\frac{\gamma(x)}{\gamma(y)}.
\end{align*}
\end{proof}
\subsection{Characterizations of norms}
An intriguing and, surprisingly, characteristic property of bisectors in Euclidean spaces is their hyperplanarity, see \cite[Proposition~4.10]{JahnMaRi2016}. Closely related, isosceles orthogonality is an homogeneous or additive exactly in Euclidean spaces.
\begin{Satz}
Let $(\RR^d,\gamma)$ be a generalized Minkowski space. The following statements are equivalent:
\begin{enumerate}[label={(\alph*)},leftmargin=*,align=left,noitemsep]
\item{The gauge $\gamma$ is a norm induced by an inner product.\label{isosceles-inner-product}}
\item{Isosceles orthogonality is right homogeneous.\label{isosceles-right-homogeneity}}
\item{Isosceles orthogonality is right additive.\label{isosceles-right-additivity}}
\item{Isosceles orthogonality is left homogeneous.\label{isosceles-left-homogeneity}}
\item{Isosceles orthogonality is left additive.\label{isosceles-left-additivity}}
\end{enumerate}
\end{Satz}
\begin{proof}
It is sufficient to show that right additivity, left additivity, right homogeneity, and left homogeneity imply that $\gamma$ is a norm. The claim then follows from \cite[Theorem~4.7, Theorem~4.8]{James1945}.

\ref{isosceles-right-homogeneity}: If  $y\perp_I x$, then $y\perp_I\lambda x$ for all $\lambda>0$ or, equivalently, $\lambda^{-1}y\in\bisec(-x,x)$ for all $\lambda>0$. Taking the limit $\lambda\to+\infty$, we obtain $0\in\bisec(-x,x)$. (Note that the bisector is a closed set by continuity of $\gamma$.) Since $x$ was chosen arbitrarily, $\gamma$ is a norm.

\ref{isosceles-left-homogeneity}: Given $x\in\RR^d\setminus\setn{0}$, there exists exactly one number $\alpha$ for which $\alpha x\perp_I x$. If $\gamma$ is not a norm, then $x$ can be chosen such that $\alpha\neq 0$. By left additivity, $\lambda\alpha x\perp_I x$ for all $\lambda>0$ which is impossible for $\abs{\lambda\alpha}>1$.

\ref{isosceles-right-additivity}, \ref{isosceles-left-additivity}: Like before, but with $\lambda\in\NN$ instead of $\lambda>0$.
\end{proof}

Within the class of normed spaces, Birkhoff orthogonality implies isosceles orthogonality if and only if the space is a Hilbert space, see \cite[Theorem~2]{Ohira1952}, \cite[(10.2)]{Amir1986}. A complementary statement is valid for normed spaces: Isosceles orthogonality implies Birkhoff orthogonality if and only if the space is a Hilbert space, see \cite[Theorem~1]{Ohira1952} and \cite[(10.9)]{Amir1986}. The same is true for generalized Minkowski spaces.
\begin{Satz}\
Let $(\RR^d,\gamma)$ be a generalized Minkowski space. 
\begin{enumerate}[label={(\alph*)},leftmargin=*,align=left,noitemsep]
\item{If Birkhoff orthogonality implies isosceles orthogonality, then $\gamma$ is a norm.\label{birkhoff-implies-isosceles}}
\item{If isosceles orthogonality implies Birkhoff orthogonality, then $\gamma$ is a norm.\label{isosceles-implies-birkhoff}}
\end{enumerate}
\end{Satz}
\begin{proof}
\ref{birkhoff-implies-isosceles}: We have $0\perp_B y$ for all $y\in\RR^d$, thus $\gamma(y)=\gamma(-y)$ for all $y\in\RR^d$.

\ref{isosceles-implies-birkhoff}: Assume that $\gamma$ is not a norm. Then there exists $y\in\RR^d$ such that $\gamma(y)\neq\gamma(-y)$. Furthermore, there is a unique point $x\in\strline{-y}{y}$ such that $x\perp_I y$, namely $x=\frac{\gamma(-y)-\gamma(y)}{\gamma(-y)+\gamma(y)}y\neq 0$. Due to the hypothesis, we have $x\perp_B\frac{\gamma(-y)+\gamma(y)}{\gamma(-y)-\gamma(y)}x$, which is impossible.
\end{proof}

\section{Final remarks}
For extending the concept of orthogonality from Euclidean space to arbitrary normed spaces, there are various alternatives each of which has its own benefits (see \cite{AlonsoBe1988b,AlonsoBe1989} for an overview of orthogonality types in normed spaces). By replacing the norm by a gauge, we translated two of the notions from normed spaces to generalized Minkowski spaces.
Apart from this extension of the geometric setting, the relaxation of the orthogonality relation itself has been approached not only in the way presented here, but also differently.
Dragomir \cite{Dragomir1991} introduced an approximate Birkhoff orthogonality relation $x\perp y$ via $\mnorm{x}\leq\mnorm{x+\lambda y}+\eps\mnorm{x}$ for all $\lambda\in\RR$. This condition is a left-homogeneous version of the one acting as the model for \hyperref[definition:orthogonality]{Definition~\ref*{definition:orthogonality}}.
Chmieli{\'{n}}ski \cite{Chmielinski2005b} discussed two approximate orthogonality relations in normed spaces defined via $\mnorm{x+\lambda y}^2\geq\mnorm{x}^2-2\eps\mnorm{x}\mnorm{\lambda y}$ and $\mnorm{x+\lambda y}\geq\sqrt{1-\eps^2}\mnorm{x}$ (for all $\lambda\in\RR$, in each case), the latter one being a reparametrization of Dragomir's condition.
Both relations are left-homogeneous and right-homogeneous as well.
In inner-product spaces, the condition $\mnorm{x+\lambda y}\geq\sqrt{1-\eps^2}\mnorm{x}$ for all $\lambda\in\RR$ is equivalent to $\abs{\skpr{y}{x}}\leq\eps\mnorm{x}\mnorm{y}$.
Due to the close relationship between orthogonality and the Cauchy--Schwarz inequality (see \hyperref[result:birkhoff-orthogonality-cauchy-schwarz-inequality-zero]{Lemma~\ref*{result:birkhoff-orthogonality-cauchy-schwarz-inequality-zero}}), the relaxed inequality $\abs{\skpr{y}{x}}\leq\eps\mnorm{x}\mnorm{y}$ itself might serve as a definition of approximate orthogonality.
Here, $0$-orthogonality is the Euclidean orthogonality (independently of the chosen norm), and $1$-orthogonality holds trivially because of the Cauchy--Schwarz inequality. (So the interesting cases which reflect the geometry of the normed space will satisfy $0<\eps<1$.)
However, since $\abs{\skpr{y}{x}}\leq\gamma(x)\gamma(y)$ is wrong in general, relaxing the Cauchy--Schwarz inequality in a multiplicative way has to be done differently in generalized Minkowski spaces.
Apart from superior and inferior semi-inner products whose gauge counterparts appear in the end of \hyperref[chap:dual-characterizations]{Section~\ref*{chap:dual-characterizations}}, Chmieli{\'{n}}ski \cite{Chmielinski2005b} and Dragomir \cite[Chapters~8-11]{Dragomir2004} linked (relaxed) Birkhoff-type orthogonality notions to general semi-inner products.
Therefore, the following questions are natural: Can one nicely extend the Dragomir--Chmieli{\'{n}}ski definitions to generalized Minkowski spaces in order to obtain similar results? What are suitable substitutes for semi-inner products?

In the interplay between orthogonality types, metric projections onto linear subspaces, the radial projection onto the unit ball, and, of course, related characterizations of special classes of Banach spaces, also several constants and moduli which describe the geometry of the unterlying space take part. Notable examples are the James constant \cite{HaoWu2011,MartiniWu2015}, the Dunkl-Williams constant \cite{Mizuguchi2013}, the rectangular constant \cite{Joly1969,KapoorMa1981,Baronti1981,BenitezRi1977,GhoshSaPa2014}, and the Sch{\"{a}}ffer--Thele constant which also coincides with the bias and the metric projection bounds of Smith, Baronti, and Franchetti \cite{Desbiens1992}. (Note that the rectangular constant and the Sch{\"{a}}ffer--Thele constant are special values of the rectangular modulus introduced in \cite{Serb1999}.) To our best knowledge, such notions have not been investigated in generalized Minkowski spaces. 

In normed spaces, isosceles orthogonality is trivially a symmetric relation. This is not the case for all other gauges. In view of \hyperref[result:bisector-directional-convexity]{Lemma~\ref*{result:bisector-directional-convexity}}, the following question has to be answered separately: For given vectors $x,y\in\RR^d$, $x\neq 0$, in a generalized Minkowski space $(\RR^d,\gamma)$, is there a number $\alpha\in\RR$ such that $x\perp_I(\alpha x+y)$?\\[\baselineskip]
\textbf{Acknowledgments.} The author would like to thank Javier Alonso for pointing out related work and for his time and thoughts.

\providecommand{\bysame}{\leavevmode\hbox to3em{\hrulefill}\thinspace}

\end{document}